\documentclass{amsart}

\newtheorem{lem}{Lemma}[section] 
\newtheorem{prp}[lem]{Proposition}
\newtheorem{thm}[lem]{Theorem} 
\newtheorem{crl}[lem]{Corollary} 
\newtheorem{defi}[lem]{Definition} 
\newtheorem{nota}[lem]{Notation}

\newtheorem{rque}[lem]{Remark}

\newtheorem{exe}[lem]{Example} 

\numberwithin{equation}{section}

\newcommand\N{\mathbb{N}}
\newcommand\Z{\mathbb{Z}}

\newcommand\Pol{\text{Pol}}
\DeclareMathOperator{\supp}{supp}
\DeclareMathOperator{\id}{id}
\DeclareMathOperator{\Lin}{Lin}

\usepackage{amsmath,amssymb,amsthm}
\usepackage{mathrsfs}
\usepackage{graphicx}
\usepackage{amsmath}
\usepackage[all]{xy}

\title[FUSION RULES OF SOME FREE WREATH PRODUCTS BY $S_N^+$ AND APPLICATIONS]{\textbf{The fusion rules of some free wreath product quantum groups and applications}}
\author{Fran\c cois Lemeux}
\address{Fran\c cois Lemeux, Laboratoire de math\'ematiques de Besan\c con, UFR Sciences et Techniques, Universit\' e de Franche-Comt\' e,16 route de Gray, 25000 Besan\c con, France}
\curraddr{}
\email{francois.lemeux@univ-fcomte.fr}
\thanks{}

\begin{document}
\maketitle

\begin{abstract}
In this article we find the fusion rules of the free wreath product quantum groups $\widehat{\Gamma}\wr_*S_N^+$ for any discrete group $\Gamma$. To do this we describe the spaces of intertwiners between basic corepresentations which allows us to identify the irreducible corepresentations. We then apply the knowledge of the fusion rules to prove, in most cases, several operator algebraic properties of the associated reduced $C^*$-algebras such as simplicity and uniqueness of the trace. We also prove that the associated von Neumann algebra is a full type $II_1$-factor and that the dual of $\widehat{\Gamma}\wr_*S_N^+$ has the Haagerup approximation property for all finite groups $\Gamma$.
\end{abstract}

\section*{Introduction}
One motivation for the construction of quantum groups was the generalization of Pontrjagin duality to non-abelian locally compact groups: if $G$ is an abelian locally compact group then the set of characters $\widehat{G}$ is an abelian locally compact group again and and the bidual is isomorphic with $G$. Of course if $G$ is not abelian, one can not expect that this latter property holds and then one has to look for a larger category, the one of quantum groups, that includes locally compact groups and their (generalized) duals. In \cite{ES75} Enock and Schwartz  and in \cite{VK74} Va{\u\i}nerman and Kac, defined the notion of Kac algebra in the setting of von Neumann algebras. The notion of the dual $\widehat{A}$ of a Kac algebra $A$ was also introduced in this article and it was proved that the bidual of such a Kac algebra is isomorphic with $A$. Kac algebras are endowed with the same structural maps as Hopf-algebras (coproduct, antipode, counit). A $C^*$-algebraic theory and analogue results were proved in this setting, see e.g. \cite{EV93}. These algebras together with their structural maps are examples of quantum groups.

In \cite{Wor87b} Woronowicz constructed a $C^*$-algebra associated to the so called twisted $SU_q(2)$ quantum groups, with a coproduct but with an unbounded antipode which is not a $*$-anti-automorphism as in the case of Kac algebras. So the category of Kac algebras appeared not large enough to contain all interesting examples of quantum groups. In \cite{Wor87}, \cite{Wor98}, Woronowicz introduced a general theory of compact quantum groups in the setting of $C^*$-algebras. This construction includes also the Drinfeld-Jimbo type quantum groups. Under minimal assumptions, the existence and uniqueness of a Haar state could be proved and a Peter-Weyl representation theory of compact quantum groups could be developed, very close to the one for (classical) compact groups. Moreover, as in the classical case the dual $\widehat{\mathbb{G}}$ of a compact quantum group $\mathbb{G}$ is a discrete quantum group i.e. the underlying $C^*$-algebra is direct sum of matrix algebras.

In \cite{Wan93} and \cite{Wang2}, Wang constructed examples of compact quantum groups $U_N^+$ and $O_N^+$ called the free unitary and free orthogonal quantum groups and introduced also, in \cite{Wang}, the quantum permutation groups $S_N^+$. We recall the definitions of the underlying Woronowicz-$C^*$-algebras:

\begin{enumerate}
\item[$\bullet$] $C(U_N^+)=C^*-\langle u_{ij} :\ 1\le,i,j\le N |\ (u_{ij})_{ij} \text{ and } (u_{ij}^*)_{ij} \text{ are unitaries}\rangle$
\item[$\bullet$] $C(O_N^+)=C^*-\langle o_{ij} :\ 1\le,i,j\le N |\ o_{ij}^*=o_{ij} \text{ and } (o_{ij})_{ij} \text{ is unitary}\rangle$
\item[$\bullet$] $C(S_N^+)=C^*-\langle v_{ij} :\ 1\le,i,j\le N |\ (v_{ij})_{ij }\text{ is a magic unitary}\rangle$
\end{enumerate}
which are ``free" versions of the commutative $C^*$-algebras of functions $C(U_N)$, $C(O_N)$, $C(S_N)$. These compact quantum groups were studied by Banica and others. He described in \cite{MR1484551}, \cite{Ban96} and \cite{Ban05}, their irreducible corepresentations and the fusion rules binding them i.e. the way that the tensor products of two irreducible corepresentations decomposes into irreducibles. This work laid the foundations for the study of the geometric, analytic and combinatoric properties of these quantum groups.


Later, new examples of compact quantum groups appeared. Banica and Speicher introduced the notion of easy quantum groups, \cite{BS09}. They are compact quantum groups whose Woronowicz-$C^*$-algebras are generated by a unitary matrix (with additional properties). Their intertwiner spaces have a combinatorial description in terms of partitions. These compact quantum groups cover the basic examples $O_N^+, S_N^+$ we mentioned above and include new ones. More recently, Weber \cite{Web13}, Raum and Weber \cite{WebR12}, Freslon and Weber \cite{FreW13}, investigated these ``combinatorial" quantum groups in order to classify them.

In \cite{Bic04}, Bichon introduced the notion of free wreath product $A*_wC(S_N^+)$ for any Woronowicz-$C^*$-algebra $A$ (see \cite{Bic04}). He proved that when $\mathbb{G}=(A,\Delta)$ is a compact quantum group, $\mathbb{G}\wr_*S_N^+=(A*_w C(S_N^+),\Delta)$ is again a compact quantum group. In \cite{BV09}, Banica and Vergnioux studied the quantum reflection groups $H_N^{s+}$, another family of compact quantum groups introduced in 2007, see \cite{BBCC11} for the definition and first properties. In particular, they proved an isomorphism of compact quantum groups $H_N^{s+}\simeq\widehat{\Z_s}\wr_*S_N^+$. 

There is no description of the fusion rules of $\mathbb{G}\wr_*S_N^+$ in general. When $\mathbb{G}$ is the dual of $\Z_s:=\Z/s\Z$ or $\Z$, we recalled above that the free wreath products $\widehat{\Z_s}\wr_*S_N^+$ correspond to $H_N^{s+}, s\in[1,\infty]$. In \cite{BV09}, Banica and Vergnioux found the fusions rules of $H_N^{s+}$ (for $s\in[1,\infty]$ and $N\ge4$). In this article, we generalize the description of the fusion rules of $\widehat{\Z_s}\wr_*S_N^+$ to the free wreath products $\widehat{\Gamma}\wr_*S_N^+$ (with the notation above, $A=C^*(\Gamma)$) for \textit{all} discrete groups $\Gamma$. This provides a whole new class of compact quantum groups with an explicit description of the fusion rules. 


Another aspect of this work is to pursue the study of the operator algebras associated to compact quantum groups. Banica started it in \cite{MR1484551} by proving the simplicity of $C_r(U_N^+)$. Vergnioux proved in \cite{Ver05O} the property of Akemann-Ostrand for $L^{\infty}(U_N^+)$ and $L^{\infty}(O_N^+)$ and together with Vaes proved the factoriality, fullness and exactess for $L^{\infty}(O_N^+)$ in \cite{VV07}. More recently, in \cite{Bra11} and \cite{Bra12}, Brannan proved the Haagerup property for $L^{\infty}(O_N^+)$, $L^{\infty}(U_N^+)$ and $L^{\infty}(S_N^+)$. Freslon proved the weak-amenability of $L^{\infty}(O_N^+)$, $L^{\infty}(U_N^+)$ in \cite{Fre13} and together with De Commer and Yamashita proved the weak amenability for $L^{\infty}(S_N^+)$ in \cite{CFY13}. In each of these results, the knowledge of the fusion rules of the compact quantum groups is a crucial tool to prove the properties of the associated reduced $C^*$-algebras and von Neumann algebras.

From the results proved in this article arise the following questions: 
\begin{enumerate}
\item[$\bullet$] Is it true that if $\Gamma$ is a discrete group with the Haagerup property, then the dual of $H_N^+(\Gamma)\simeq\widehat{\Gamma}\wr_*S_N^+$ has the Haagerup property ?
\item[$\bullet$] Is it true that the dual of $\widehat{\Gamma}\wr_*S_N^+$ is weakly-amenable ?
\item[$\bullet$] Which other algebraic operator properties possess $C_r(H_N^+(\Gamma))$, $L^{\infty}(H_N^+(\Gamma))$ (bi-exactness, property RD etc.) ?
\item[$\bullet$] Can one compute the fusion rules of wreath products $\mathbb{G}\wr_*S_N^+$ if $\mathbb{G}$ is a compact quantum group with known fusion rules ?
\end{enumerate}

This article is organized as follows. The first section is dedicated to recall some definitions and general results on compact quantum groups. It contains the definitions of the quantum permutation groups $S_N^+$, of the free wreath products $H_N^+(\Gamma)\simeq \widehat{\Gamma}\wr_*S_N^+$ and the special case of quantum reflection groups.

In the second section, we recall the concept of Tannaka-Krein duality and we describe the intertwiner spaces of certain basic corepresentations in $H_N^+(\Gamma)$ using a canonical homomorphism from the universal algebra of a certain free product of compact quantum groups onto $C(H_N^+(\Gamma))$. This allows us to compute the fusion rules between the irreducible corepresentations of $H_N^+(\Gamma)$.

In the third section, we give several applications of this description of the fusion rules:
\begin{enumerate}
\item[$\bullet$] We prove the simplicity and the uniqueness of the trace of the reduced $C^*$-algebra $C^*_r(H_N^+(\Gamma))$ for all discrete groups $|\Gamma|\ge2$ and all $N\ge8$ (in particular $L^{\infty}(H_N^+(\Gamma))$ is a $II_1$-factor). We adapt a variant of Powers' methods used by Banica in \cite{MR1484551} and we use the simplicity of $C_r(S_N^+)$ for all $N\ge8$ proved in \cite{Bra12}.
\item[$\bullet$] We show the fullness of the $II_1$-factor $L^{\infty}(H_N^+(\Gamma))$ for all $N\ge8$ and any discrete group $\Gamma$. The proof is adapted from the ``$14-\epsilon$ method" which is used in the classical proof of the fact that $L^{\infty}(F_n)$ does not have the property $\Gamma$. This application is based on work by Vaes for the fullness of $L^{\infty}(U_N^+)$ which can be found in an appendix to \cite{CFY13}.
\item[$\bullet$] We finish our article by extending the main result in \cite{Lem13} by proving that the duals of the free wreath products $\widehat{\Gamma}\wr_*S_N^+$ have the Haagerup property for all \textit{finite} groups $\Gamma$ and all $N\ge4$.
\end{enumerate}

\section{Preliminaries}\label{preliminaries}
\subsection{Quantum groups and representation theory}
The definition and basics on compact quantum groups go back to Woronowicz. In this first section we recall a few facts and results about compact quantum groups. One can refer to \cite{Wor98} for a nive survey of the notions needed in this article. We also recommend the book \cite{Tim08}.

A compact quantum group is a pair $\mathbb{G}=(C(\mathbb{G}),\Delta)$ where $C(\mathbb{G})$ is a unital Woronowicz-$C^*$-algebra: $C(\mathbb{G})$ is a unital $C^*$-algebra endowed with a unital $*$-homorphism $\Delta : C(\mathbb{G})\to C(\mathbb{G})\otimes_{\min} C(\mathbb{G})$ which is coassociative $(\text{id}\otimes\Delta)\circ\Delta=(\Delta\otimes \text{id})\circ\Delta$ and such that we have the cancellation property, i.e. $\text{span}\{\Delta(a)(b\otimes 1) : a,b\in C(\mathbb{G})\}$ and $\text{span}\{\Delta(a)(1\otimes b) : a,b\in C(\mathbb{G})\}$ are norm dense in $C(\mathbb{G})\otimes C(\mathbb{G})$. These assumptions allow to prove the existence and uniqueness of a Haar state $h : C(\mathbb{G})\to \mathbb{C}$ satisfying the bi-invariance relations $(h\otimes \text{id})\circ\Delta(\cdot)=h(\cdot)1=(\text{id}\otimes h)\circ\Delta(\cdot)$, see \cite[Theorem 1.3]{Wor98}. In this article we will deal with compact quantum groups of Kac type, that is their Haar state $h$ is a trace. We will also deal with several compact \textit{matrix} quantum groups, namely the quantum permutation groups $S_N^+$ and the quantum reflection groups $H_N^{s+}$, see Section \ref{qpgrappel} and \ref{qrgrappel}. A compact matrix quantum group $\mathbb{G}=(C(\mathbb{G}),\Delta)$ is a compact quantum group whose underling Woronowicz-$C^*$-algebra $C(\mathbb{G})$ is generated by the coefficients of a generating unitary matrix with coefficients in $C(\mathbb{G})$, see \cite{Wor87}.

One can consider the GNS representation $\lambda_h : C(\mathbb{G})\to B(L^2(\mathbb{G},h))$ associated to the Haar state $h$ of $\mathbb{G}=(C(\mathbb{G}),\Delta)$ called the left regular representation. We will denote by $\Lambda_h$ the GNS map $\Lambda_h : C(\mathbb{G})\to L^2(\mathbb{G},h)$. The reduced $C^*$-algebra associated to $\mathbb{G}$ is then defined by $C_r(\mathbb{G})=\lambda_h(C(\mathbb{G}))\simeq C(\mathbb{G})/\ker(\lambda_h)$ and the von Neumann algebra by $L^{\infty}(\mathbb{G})=C_r(\mathbb{G})''$. One can prove that $C_r(\mathbb{G})$ is again a Woronowicz-$C^*$-algebra whose Haar state extends to $L^{\infty}(\mathbb{G})$. We will denote simply by $\Delta$ and $h$ the coproduct and Haar state on $C_r(\mathbb{G})$. 

An $N$-dimensional (unitary) corepresentation $u=(u_{ij})_{ij}$ of $\mathbb{G}$ is a (unitary) matrix $u\in M_N(C(\mathbb{G}))\simeq C(\mathbb{G})\otimes B(\mathbb{C}^N)$ such that for all $i,j\in\{1,\dots,N\}$, one has $$\Delta(u_{ij})=\sum_{k=1}^Nu_{ik}\otimes u_{kj}.$$ 
An intertwiner between two corepresentations $$u\in M_{N_u}(C(\mathbb{G})) \text{ and } v\in M_{N_v}(C(\mathbb{G}))$$ is a matrix $T\in M_{N_u,N_v}(\mathbb{C})$ such that $v(1\otimes T)=(1\otimes T)u$.
We say that $u$ is equivalent to $v$, and we note $u\sim v$, if there exists an invertible intertwiner between $u$ and $v$. We denote by $\text{Hom}_{\mathbb{G}}(u,v)$ the space of intertwiners between $u$ and $v$. A corepresentation $u$ is said to be irreducible if $\text{Hom}_{\mathbb{G}}(u,u)=\mathbb{C} \text{id}$. We denote by $\text{Irr}(\mathbb{G})$ the set of equivalence classes of irreducible corepresentations of $\mathbb{G}$.

We recall that $C(\mathbb{G})$ contains a dense $*$-subalgebra denoted by $\Pol(\mathbb{G})$ which is linearly generated by the coefficients of the irreducible corepresentations of $\mathbb{G}$. The coefficients of a corepresentation $u$ of $\mathbb{G}$ acting on a Hilbert space $H_u$ are given by $(\text{id}\otimes\phi)(u)$ for some functional $\phi\in B(H_u)^*$. This algebra has a Hopf-$*$-algebra structure and in particular there exists a $*$-antiautomorphism $\kappa : \Pol(\mathbb{G})\to \Pol(\mathbb{G})$, called the antipode, which acts on the coefficients of the irreducible unitary corepresentations by $\kappa(u_{ij})=u_{ji}^*$.

If $u$ is a corepresentation of $\mathbb{G}$ acting on a Hilbert space of dimension $N$, the matrix $\overline{u}=(u_{ij}^*)\in M_N(C(\mathbb{G}))$ is again a corepresentation of $\mathbb{G}$. It is called the conjugate of $u$. In general it is not necessarily unitary, even if $u$ is. Recall that $u$ is unitary if and only if on each coefficient, we have $\kappa(u_{ij})=u_{ji}^*$, see e.g. \cite[Proposition 3.1.7]{Tim08}. Note that all the compact quantum groups we will deal with are of Kac type. In this case, the conjugate of a unitary corepresentation is also unitary since in this case the antipode satisfies $\kappa^2=\text{id}$. One can refer to \cite{BS93} for the proofs of the equivalence of the Kac characterizations we gave, namely the fact that the Haar state is a trace and the fact that $\kappa^2=\text{id}$.

The algebra $Pol(\mathbb{G})$ algebra is also dense in $L^2(\mathbb{G},h)$. Since $h$ is faithful on the $*$-algebra $\Pol(\mathbb{G})$, one can identify $\Pol(\mathbb{G})$ with its image in the GNS-representation $\lambda_h(C(\mathbb{G}))$. We will denote by $\chi_r$ the character of the irreducible corepresentation $r\in \text{Irr}(\mathbb{G})$, that is $\chi_r=(\text{id}\otimes \text{Tr})(r)$. 

In the Kac type case, the right regular representation of a compact quantum groups $\mathbb{G}=(C(\mathbb{G},\Delta)$, $\rho_h : C(\mathbb{G})\to B(L^2(\mathbb{G},h))$ is given by $\rho_h(x)\Lambda_h(y)=\Lambda_h(y\kappa(x))$. It commutes with $\lambda_h$ and one can consider the adjoint representation on $Pol(\mathbb{G})$, $$(\lambda_h,\rho_h)\circ\Delta : \Pol(\mathbb{G})\to B(L^2(\mathbb{G},h)).$$ This representation acts on the irreducible characters as follows $$\text{ad}(\chi_r)(z)=\sum_{ij}r_{ij}z\kappa(r_{ji})=\sum_{ij}r_{ij}zr_{ij}^*.$$ Notice that the map $z\mapsto \text{ad}(\chi_r)(z)$ is completely positive for all $r\in \text{Irr}(\mathbb{G})$.



In \cite{MR1484551}, Banica uses the notion of support of an element $x\in \Pol(\mathbb{G})$. The support of $x\in \Pol(\mathbb{G})$ is denoted $\supp(x)$ and defined as the smallest subset $G\subset \text{Irr}(\mathbb{G})$ such that $x$ is a linear combination of certain coefficients of elements $r\in G$. In other words, $$r\notin \supp(x)\Leftrightarrow h(xr_{ij}^*)=0, \text{ for all coefficients } r_{ij} \text{ of } r.$$

\subsection{Quantum permutation groups}\label{qpgrappel}
A fundamental and basic family of examples of compact quantum groups was introduced by Wang. It is recalled in the following definition:

\begin{defi}\label{permw}(\cite{Wang})
Let $N\ge2$. $S_N^+$ is the compact quantum group $(C(S_N^+),\Delta)$ where  
$C(S_N^+)$ is the universal $C^*$-algebra generated by $N^2$ elements $v_{ij}$ such that the matrix $v=(v_{ij})$ is unitary and $v_{ij}=v_{ij}^*=v_{ij}^2, \forall i,j$ (i.e. $v$ is a magic unitary) and such that
the coproduct $\Delta$ is given by the usual relations making of $v$ a finite dimensional corepresentation of $C(S_N^+)$, that is $\Delta(v_{ij})=\sum_{k=1}^Nv_{ik}\otimes v_{kj}$, $\forall i,j$.
\end{defi}

In the cases $N=2,3$, one obtains the usual algebras $C(\Z_2), C(S_3)$ since a magic unitary of size $2$ (respectively 3) is composed of commuting projections as one can see using the Fourier transformation over $\Z_2$, resp. $\Z_3$.  
If $N\ge4$, one can find an infinite dimensional quotient of $C(S_N^+)$ so that $C(S_N^+)$ is not isomorphic to $C(S_N)$, see e.g. \cite{Wang}, \cite{Ban05}.

The corepresentation theory of $S_N^+$ was found by Banica. It is recalled in the following theorem:

\begin{thm}(\cite{Ban99})
There is a maximal family $\left(v^{(t)}\right)_{t\in\N}$ of pairwise inequivalent irreducible finite dimensional unitary representations of $S_N^+$ such that: 
\begin{enumerate}
\item $v^{(0)}$ is the trivial representation and the (fundamental) corepresentation $v$ is equivalent to $v^{(0)}\oplus v^{(1)}$.
\item The contragredient of any $v^{(t)}$ is equivalent to itself that is $\overline{v^{(t)}}\simeq v^{(t)},\ \forall t\in\N$.
\item The fusion rules are the same as for $SO(3)$: $$v^{(s)}\otimes v^{(t)}\simeq\bigoplus_{k=0}^{2\emph{min(s,t)}}v^{(s+t-k)}$$
\end{enumerate}
\end{thm}

\subsection{Free (wreath) product quantum groups}\label{qrgrappel}
In \cite{Wang2}, Wang defined the free product $\mathbb{G}=\mathbb{G}_1*\mathbb{G}_2$ of compact quantum groups. He showed that such a free product is still a compact quantum group and gave a description of the irreducible corepresentations of $\mathbb{G}$ as alternating tensor products of nontrivial irreducible corepresentations:

\begin{thm}\label{frprodcrucial}(\cite{Wang2}) Let $\mathbb{G}_1$ and $\mathbb{G}_2$ be compact quantum groups. Then the set $\emph{Irr}(\mathbb{G})$ of irreducible corepresentations of the free product of quantum groups $\mathbb{G}=\mathbb{G}_1*\mathbb{G}_2$ can be identified with the set of alternating words in $\emph{Irr}(\mathbb{G}_1)*\emph{Irr}(\mathbb{G}_2)$ and the fusion rules can be recursively described as follows:
\begin{enumerate}
\item[$\bullet$] If the words $x,y\in \emph{Irr}(\mathbb{G})$ end and start in $\emph{Irr}(\mathbb{G}_i)$ and $\emph{Irr}(\mathbb{G}_j)$ respectively with $j\ne i$ then $x\otimes y$ is an irreducible corepresentation of $\mathbb{G}$ corresponding to the concatenation $xy\in \emph{Irr}(\mathbb{G})$.
\item[$\bullet$] If $x=vz$ and $y=z'w$ with $z,z'\in \emph{Irr}(\mathbb{G}_i)$ then $$x\otimes y=\bigoplus_{1\ne t\subset z\otimes z'} vtw\oplus \delta_{\overline{z},z'}(v\otimes w)$$ where the sum runs over all non-trivial irreducible corepresentations $t\in \emph{Irr}(\mathbb{G}_i)$ contained in $z\otimes z'$.
\end{enumerate}
\end{thm}

We will use this result to describe the fusion rules of another type of product of compact quantum groups: the free wreath products $\widehat{\Gamma}\wr_* S_N^+$, with $\Gamma$ a discrete group (see section \ref{partfusrules}). Note that the previous theorem will be crucial in the proof of Theorem \ref{FRFP} where we describe certain intertwiner spaces in the free product of compact matrix quantum groups.

 Let us now recall the definition of the free wreath products by the quantum permutation groups $S_N^+$. We denote by $\nu_i$ the canonical homomorphism $\nu_i: A\to A^{*N}$ sending $A$ to the $i$-th copy of $A$ in $A^{*N}$. The following definition is due to Bichon: 

\begin{defi}(\cite[Definition 2.2]{Bic04})
Let $A$ be a Woronowicz-$C^*$-algebra and $N\ge2$. The free wreath product of $A$ by $S_N^+$ is the quotient of the $C^*$-algebra $A^{*N}*C(S_N^+)$ by the two-sided ideal generated by the elements $$\nu_k(a)v_{ki}-v_{ki}\nu_k(a), \ \ \ 1\le i,k\le N, \ \ a\in A.$$ It is denoted by $A*_wC(S_N^+)$.
\end{defi}

In the next result, we use the Sweedler notation $\Delta_A(a)=\sum a_{(1)}\otimes a_{(2)}\in A\otimes A$.

\begin{thm}(\cite[Theorem 2.3]{Bic04})
Let $A$ be a Woronowicz-$C^*$-algebra, then free wreath product $A*_wC(S_N^+)$ admits a Woronowicz-$C^*$-algebra structure: If $a\in A$, then
$$\Delta(v_{ij})=\sum_{k=1}^Nv_{ik}\otimes v_{kj}, \forall i,j\in\{1,\dots,N\},$$ $$\Delta(\nu_i(a))=\sum_{k=1}^N\nu_i(a_{(1)})v_{ik}\otimes\nu_k(a_{(2)}) \emph{ if } \Delta_A(a)=\sum a_{(1)}\otimes a_{(2)}\in A\otimes A,$$
$$\epsilon(v_{ij})=\delta_{ij},\ \epsilon(\nu_i(a))=\epsilon_A(a),\ S(v_{ij})=v_{ji},\ S(\nu_i(a))=\sum_{k=1}^N\nu_k(S_A(a))v_{ki},$$
$$v_{ij}^*=v_{ij},\ \nu_i(a)^*=\nu_i(a^*).$$ In particular, if $\mathbb{G}$ is a compact quantum group, then $\mathbb{G}\wr_*S_N^+=(A*_wC(S_N^+),\Delta)$ is also a compact quantum group.
\end{thm}

The following examples are fundamental for the rest of this article.
\begin{exe}\label{exefond}(\cite[Example 2.5]{Bic04})
Let $\Gamma$ be a discrete group with neutral element $e$ and let $N\ge2$. Let $A_N(\Gamma)$ be the universal $C^*$-algebra with generators $a_{ij}(g), 1\le i,j\le N, g\in\Gamma$ together with the following relations:
\begin{equation}\label{tt1}
a_{ij}(g)a_{ik}(h)=\delta_{jk}a_{ij}(gh)\ \ \ ;\ \ \ a_{ji}(g)a_{ki}(h)=\delta_{jk}a_{ji}(gh)
\end{equation}
\begin{equation}\label{tt2}
\sum_{l=1}^Na_{il}(e)=1=\sum_{l=1}^Na_{li}(e),
\end{equation}
and involution 
\begin{equation}\label{tt3}
a_{ij}(g)^*=a_{ij}(g^{-1}).
\end{equation}
Then $H_N^+(\Gamma):=(A_N(\Gamma),\Delta)$ is a compact quantum group with:
\begin{equation}\label{tt3}
\Delta(a_{ij}(g))=\sum_{k=1}^N a_{ik}(g)\otimes a_{kj}(g).
\end{equation}

We have for all $g\in\Gamma$, $\epsilon(a_{ij}(g))=\delta_{ij}$ and $S(a_{ij}(g))=a_{ji}(g^{-1})$.
Furthermore, $H_N^+(\Gamma)$ is isomorphic, as a compact quantum group, with $\widehat{\Gamma}\wr_* S_N^+$. Consider the following important special cases:
\begin{enumerate}
\item If $\Gamma=\Z_s$ for an integer $s\ge 1$, one gets the quantum reflection groups $H_N^{s+}$, see \cite{BBCC11} and \cite{BV09}. $C(H_N^{s+})$ is the universal $C^*$-algebra generated by $N^2$ normal elements $U_{ij}$ such that:
\begin{enumerate}
\item $U=(U_{ij})$ and $^tU=(U_{ji})$ are unitary,
\item $U_{ij}U_{ij}^*$ is a projection, $\forall 1\le i,j\le N$,
\item\label{s fini} $U_{ij}^s=U_{ij}U_{ij}^*$, $\forall 1\le i,j\le N$,
\item $\Delta(U_{ij})=\sum_{k=1}^NU_{ik}\otimes U_{kj}$, $\forall 1\le i,j\le N$.
\end{enumerate}
\item If $\Gamma=\Z$, one gets $H_{N}^{\infty+}=(C(H_N^{\infty+}),\Delta)$  where $C(H_N^{\infty+})$ and $\Delta$ are defined as above except that one removes the relations (\ref{s fini}) above.
\end{enumerate}
\end{exe}

\begin{rque}
Notice that if $\Gamma$ has cardinal $|\Gamma|=1$, then $H_N^+(\Gamma)=S_N^+$. 
\end{rque}

One can consider other examples of free wreath product quantum groups $\widehat{\Gamma}\wr_*S_N^+$, for instance:
\begin{exe}
Let $N\ge2$ and let $\mathbb{F}_2=\langle a,b\rangle$ be the (classical) non-abelian free group on two generators. Then the underlying Woronowicz-$C^*$-algebra of the free wreath product quantum group $\widehat{\mathbb{F}_2}\wr_*S_N^+$ is generated by $2N^2$ generators $$U_{ij}^{(a)}=\nu_i(u_a)v_{ij}, U_{ij}^{(b)}=\nu_i(u_b)v_{ij}\ \forall1\le i,j\le N$$ such that:
\begin{enumerate}
\item[(a)] $U^{(a)}=(U_{ij}^{(a)})$, $^tU^{(a)}=\left(U_{ji}^{(a)}\right)$, $U^{(b)}=\left(U_{ij}^{(b)}\right)$ and $^tU^{(b)}=(U_{ji}^{(b)})$ are unitary,
\item[(b)] $U_{ij}^{(a)}\left(U_{ij}^{(a)}\right)^*$, $U_{ij}^{(b)}\left(U_{ij}^{(b)}\right)^*$ are projections, $\forall 1\le i,j\le N$,
\item[(c)] $\Delta\left(U_{ij}^{(a)}\right)=\sum_{k=1}^NU_{ik}^{(a)}\otimes U_{kj}^{(a)}$, $\Delta\left(U_{ij}^{(b)}\right)=\sum_{k=1}^NU_{ik}^{(b)}\otimes U_{kj}^{(b)}$, $\forall 1\le i,j\le N$.
\end{enumerate}
\end{exe}


We will use the following proposition to find the fusion rules of $H_N^+(\Gamma)$.
\begin{prp}\label{arrows}
Let $\Gamma$ be a finitely generated group $\Gamma=\langle g_1,\dots,g_p\rangle$ and denote by $s_r\in[1,+\infty]$ the order of the element $g_r$. We have canonical homomorphisms $$*_{r=1}^pC(H_N^{s_r+})\overset{\pi_1}{\longrightarrow} C(H_N^+(\Gamma))\overset{\pi_2} {\longrightarrow} C(S_N^+)$$ given, for all $k\in\mathbb{N}$, $1\le r_t\le p$, $1\le i_t,j_t\le N$, by $$\pi_1\left(U_{i_1j_1}^{(r_1)}\dots U_{i_kj_k}^{(r_k)}\right)=a_{i_1j_1}(g_{r_1})\dots a_{i_kj_k}(g_{r_k}),$$ and for all $1\le i,j\le N, g\in\Gamma$ by $$\pi_2(a_{ij}(g))=v_{ij},$$ where $U_{ij}^{(r)}$ is the coefficient of the fundamental corepresentation of $C(H_N^{s_r+})$ seen as the $r$-th factor of the free product $*_{r=1}^pC(H_N^{s_r+})$. 
\end{prp}
\begin{proof}
The existence of the homomorphism $\pi_2$ is clear by universality of $C(H_N^+(\Gamma))$. For $\pi_1$, notice that for all $1\le r\le p$ there is, by universality of $C(H_{N}^{s_r+})$, a homomorphism $\pi_1^{(r)}$ such that $\pi_1^{(r)}(U_{ij}^{(r)})=a_{ij}(g_r)$ for all $1\le i,j\le N$. Then $$\pi_1:=*_{r=1}^p\pi_1^{(r)}: *_{r=1}^pC(H_N^{s_r+})\to C(H_N^+(\Gamma))$$ satisfies for all $1\le r_t\le p$, $1\le i_t,j_t\le N$, $$\pi_1\left(U_{ij}^{(r_1)}\dots U_{ij}^{(r_k)}\right)=\pi_{1}^{(r_1)}\left(U_{ij}^{(r_1)}\right)\dots\pi_{1}^{(r_k)}\left(U_{ij}^{(r_k)}\right)=a_{i_1j_1}(g_{r_1})\dots a_{i_kj_k}(g_{r_k}).$$
\end{proof}

In the next theorem, we will recall the description of the corepresentations and fusion rules proved by Banica and Vergnioux in \cite{BV09} for the compact quantum reflection groups $H_{N}^{s+}, H_N^{\infty+}$ ($N\ge4$). In the case $s=\infty$ we make the convention that $\Z_s=\Z$. 

\begin{defi}
Let $F=\langle\Gamma\rangle$ (\/$\Gamma=\Z_s$, $s\in[1,+\infty]$) be the monoid formed by the words over the group $\Gamma$, endowed with the operations 
\begin{enumerate}
\item Involution: $(i_1\dots i_k)^-=(-i_k)\dots(-i_1)$,
\item Fusion: $(i_1\dots i_k).(j_1\dots j_l)=i_1\dots i_{k-1}(i_k+j_1)j_2\dots j_l$.
\end{enumerate}
\end{defi}

\begin{thm}\label{basicsAHS} Let $N\ge4$, $s\in[1,\infty]$.
$C(H_N^{s+})$ has a unique family of $N$-dimensional corepresentations (called basic corepresentations) $\{U_k: k\in\Z\}$ satisfying the following conditions:
\begin{enumerate}
\item $U_k=(U_{ij}^k)$ for any $k>0$.
\item $U_k=U_{k+s}$ for any $k\in\Z$.
\item $\overline{U}_k=U_{-k}$ for any $k\in\Z$.
\item $U_k, k\ne0$ are irreducible.
\item $U_0=1\oplus\rho_0$, $\rho_0$ irreducible.
\item $\rho_0,U_1,\dots,U_{s-1}$ are inequivalent corepresentations.
\end{enumerate} 
Furthermore if we write for all $i\in\Z_s$, $\rho_i=U_i\ominus \delta_{i0}1$, then the irreducible corepresentations of $C(H_N^{s+})$ can be labelled $\rho_x, x\in \langle\Z_s\rangle$ and the involution and fusion rules are $\overline{\rho}_x=\rho_{\overline{x}}$ and 
$$\rho_x\otimes\rho_y=\sum_{x=vz,\ y=\overline{z}w}\rho_{vw}\ \oplus \displaystyle\sum_{\substack{x=vz,\ y=\overline{z}w\\ v\ne\emptyset,w\ne\emptyset}} \rho_{v.w}.$$
\end{thm}

We want to prove that one can generalize the description of the irreducible corepresentations and fusion rules to $H_N^+(\Gamma)=\widehat{\Gamma}\wr_*S_N^+$ for any discrete group $\Gamma$, $N\ge4$. We will then get several operator algebraic properties that one can deduce from the knowledge of these fusion rules.

We recall what is already known in the case $\Gamma=\Z_s$ ($s\ge1$ finite):
\begin{thm}(\cite{Lem13})
The dual of $H_N^{s+}$ has the Haagerup property for all $N\ge4, s\in[1,+\infty)$.
\end{thm}

\section{Fusion rules for $H_N^+(\Gamma)$}\label{partfusrules}
In this section $\Gamma$ is any discrete group,
$N$ is an integer $N\ge4$. We are going to describe the irreducible corepresentations of $H_N^+(\Gamma)=\widehat{\Gamma}\wr_*S_N^+$ and the fusion rules binding them.  To fulfill this, we are going to use techniques introduced in \cite{Ban96}, \cite{MR1484551}, \cite{Ban99} and developed in \cite{BBCC11} and \cite{BV09}. 

Let $\mathbb{G}=(C(\mathbb{G}),\Delta)$ be a compact quantum group such that $C(\mathbb{G})$ is generated by the coefficients of a certain family $(v_i)_{i\in I}$ of finite dimensional corepresentations of $\mathbb{G}$. Denote by $\text{Rep}(\mathbb{G})$ the complete monoidal $C^*$-category with conjugates of all finite dimensional corepresentations of $\mathbb{G}$ (see e.g. \cite{Wor88} and \cite{Nesh} for the definitions of such rigid monoidal categories). We will keep the following notation in the sequel of this article:
\begin{nota}\label{TENS}
We denote by $\emph{Tens}(\mathbb{G},(v_i)_{i\in I})\subset \emph{Rep}(\mathbb{G})$ the full tensor subcategory with: 
\begin{enumerate}
\item[-] objects: all the tensor products between corepresentations $v_i$ and $\overline{v_i}$,
\item[-] morphisms: intertwiners between such tensor products.
\end{enumerate}

$\emph{Tens}(\mathbb{G},(v_i)_{i\in I})$ is contained in the category of (all) linear maps between tensor products of the representation spaces $H_i$ of the corepresentations $v_i$. Denote this category $\emph{Vect}(H_i)$.
\end{nota}
We will denote by $v$ the generating matrix of $S_N^+$ and denote by $a_{kl}(\gamma_i)$ the generating elements of $H_N^+(\Gamma)$.
Notice that, if $U^{(i)}$ denotes the fundamental corepresentation of $H_N^{s_i+}$ in $*_{i=1}^p(H_N^{s_i+})$, then the homomorphisms of Proposition \ref{arrows} at the level of the objects $$U^{(i)}\mapsto (a_{kl}(\gamma_i))_{k,l}\mapsto v,$$ give functors at the level of the categories
\begin{equation}\label{inclusion}
\text{Tens}\left(*_{i=1}^p(H_N^{s_i+}),\left\{U^{(i)}\right\}_{i=1}^p\right)\to \text{Tens}\left(H_N^+(\Gamma),a(\gamma_i): i=1,\dots,p\right)\to \text{Tens}(S_N^+,v).
\end{equation}

It is known (and recalled in the next subsection) that the tensor categories $\text{Tens}(H_N^{s+},U)$ and $\text{Tens}(S_N^+,v)$ have a diagrammatic description in terms of non-crossing partitions:  morphisms (i.e. intertwiners between corepresentations) can be described by (certain) non-crossing partitions. We use this fact to obtain a diagrammatic description of the tensor category $\text{Tens}(H_N^+(\Gamma), a(\gamma_i) : i=1,\dots,p)$. 
The inclusions (\ref{inclusion}) above, together with the diagrammatic description of $$\text{Tens}\left(*_{i=1}^p(H_N^{s_i+}),\left\{U^{(i)}\right\}_{i=1}^p\right)$$ which we will obtain in the subsection \ref{fpfusrules}, will allow us to draw our conclusions.

Before investigating these questions, we recall that the fusion rules of $H_N^+(\Gamma)$ are known in the case $N=2$, see \cite{Bic04}. In the sequel, we assume that $N\ge4$. The case $N=3$ is not investigated in this article since a crucial result in our argument (namely Theorem \ref{Tpindep} $(3)$) is not true in this case.

\subsection{Non-crossing partitions, diagrams. Tannaka-Krein duality}
In the following paragraph, we recall a few notions on non-crossing partitions, see e.g. \cite{BV09} for more informations.
\begin{defi}
A non-crossing partition of a set with repetitions $\{1,\dots,k,k+1,\dots,k+l\}$ with $k+l$ ordered elements $1<\dots<k<k+1<\dots<k+l$, is a picture of the following form:
\smallskip 
\[ \left\{ \begin{array}{ccccc}
1\ &  &  & \ & k\\
.\ & . & . & .\ & .\\
& &  \mathscr{P} & & \\
. & . & . & . &   \\
k+1 &  &  & k+l &   \\
\end{array} \right\}. \]
It contains $k$ upper points, $l$ lower points and consists of a diagram $\mathscr{P}$ composed of strings which connect certain upper and/or lower points and which do not cross one another. We denote by $NC(k,l)$ the set of all non-crossing partitions between $k$ upper points and $l$ lower points.
\end{defi}

Such non-crossing partitions give rise to new ones by tensor product, composition and involution:

\begin{defi}\label{NC} Let $p\in NC(l,m)$ and $q\in NC(k,l)$. Then, the tensor product, composition and involution of the partitions p,q are obtained by horizontal concatenation, vertical concatenation and upside-down turning:
$$p\otimes q=\left\{\mathscr{P}\mathscr{Q}\right\},\ 
pq=\setlength{\unitlength}{0.5cm}
\left\{\ \begin{picture}(1,1)\thicklines
\put(0,0.5){$\mathscr{Q}$}
\put(0,-0.6){$\mathscr{P}$}
\end{picture}\ \right\}-\{\emph{middle points, closed blocks}\},\ 
p^{*}=\{\mathscr{P^{\downarrow}}\}.$$
The composition $pq$ is only defined if the number of lower points of $q$ is equal to the number of upper points of $p$. When one identifies the lower points of $q$ with the upper points of $p$, closed blocks might appear, that is strings which are connected neither to the new upper points nor to the new lower points. These blocks are discarded from the final pictorial representation. We will denote by $b(p,q)$ the number of closed blocks appearing when performing a vertical concatenation.
\end{defi}

\begin{exe}
Following the rules stated above (erasing middle points, discarding closed blocks and following the lines when one identifies the upper points of $p$ with the lower points of $q$), we get 
\setlength{\unitlength}{0.5cm}
$$if\ \ \ 
p=\left\{\ \begin{picture}(3,3)\thicklines
\put(0,1.2){\line(0,1){0.75}}
\put(2,1.2){\line(0,1){0.75}}
\put(1,1.2){\line(0,1){0.75}}
\put(0,1.228){\line(1,0){1}}
\put(2,1.228){\line(1,0){1}}
\put(3,1.2){\line(0,1){0.75}}

\put(0,-0.878){\line(1,0){1}}
\put(0,-1.6){\line(0,1){0.75}}
\put(1,-1.6){\line(0,1){0.75}}
\put(2,-1.6){\line(0,1){0.75}}

\put(-0.2,2.2){1}
\put(0.8,2.2){2}
\put(1.8,2.2){3}
\put(2.8,2.2){4}
\put(-0.2,-2.2){1}
\put(0.8,-2.2){2}
\put(1.8,-2.2){3}
\end{picture}\ \right\} \ \ \ \ and \ \ \
q=\left\{\ \begin{picture}(4,3)\thicklines
\put(0,-1.3){\line(0,1){3.2}}
\put(2,1.2){\line(0,1){0.75}}
\put(4,1.2){\line(0,1){0.75}}
\put(1,-1.3){\line(0,1){3.2}}
\put(2,1.228){\line(1,0){1}}
\put(3,1.2){\line(0,1){0.75}}

\put(2,-1.6){\line(0,1){0.75}}
\put(3,-1.6){\line(0,1){0.75}}
\put(2,-0.878){\line(1,0){1}}

\put(-0.2,2.2){1}
\put(0.8,2.2){2}
\put(1.8,2.2){3}
\put(2.8,2.2){4}
\put(3.8,2.2){5}
\put(-0.2,-2.2){1}
\put(0.8,-2.2){2}
\put(1.8,-2.2){3}
\put(2.8,-2.2){4}
\end{picture}\ \right\}\ \ \ \ then\ \ \
pq=\left\{\ \begin{picture}(4,3)\thicklines
\put(0,1.2){\line(0,1){0.75}}
\put(2,1.2){\line(0,1){0.75}}
\put(1,1.2){\line(0,1){0.75}}
\put(0,1.228){\line(1,0){1}}
\put(2,1.228){\line(1,0){1}}
\put(3,1.2){\line(0,1){0.75}}
\put(4,1.2){\line(0,1){0.75}}

\put(0,-0.878){\line(1,0){1}}
\put(0,-1.6){\line(0,1){0.75}}
\put(1,-1.6){\line(0,1){0.75}}
\put(2,-1.6){\line(0,1){0.75}}

\put(-0.2,2.2){1}
\put(0.8,2.2){2}
\put(1.8,2.2){3}
\put(2.8,2.2){4}
\put(3.8,2.2){5}
\put(-0.2,-2.2){1}
\put(0.8,-2.2){2}
\put(1.8,-2.2){3}
\end{picture}\ \right\}.
$$
\end{exe}
\bigskip\noindent

From non-crossing partitions $p\in NC(k,l)$ naturally arise linear maps $T_p: \mathbb{C}^{N^{\otimes k}}\to \mathbb{C}^{N^{\otimes l}}$.
\begin{defi}
Let $(e_i)$ be the canonical basis of $\mathbb{C}^N$ and $p\in NC(k,l)$ be a non-crossing partition. Then $T_p\in B\left(\mathbb{C}^{N^{\otimes k}},\mathbb{C}^{N^{\otimes l}}\right)$ is defined by:
$$T_p(e_{i_1}\otimes\dots\otimes e_{i_k})=\sum_{j_1,\dots,j_l}\delta_p(i,j)e_{j_1}\otimes\dots\otimes e_{j_l}$$
where $i$ (respectively $j$) is the $k$-tuple $(i_1,\dots,i_k)$ (respectively the $l$-tuple $(j_1,\dots,j_l)$) and if one decorates the points of $p$ by $i$ and $j$, $\delta_p(i,j)$ is equal to:
\begin{enumerate}
\item $1$ if all strings of $p$ join equal indices,
\item $0$ otherwise.
\end{enumerate}
\end{defi}

\begin{exe}
We consider an element $p\in NC(4,3)$, choose any tuples $i=(i_1,i_2,i_3,i_4)$ and $j=(j_1,j_2,j_3)$, and put them on the diagram:
\setlength{\unitlength}{0.5cm}
$$
p=\left\{\ \begin{picture}(3.1,3.5)\thicklines
\put(0,0.5){\line(0,1){1.5}}
\put(2,1){\line(0,1){1}}
\put(1,0.5){\line(0,1){1.5}}
\put(0,0.5){\line(1,0){3}}
\put(3,0.5){\line(0,1){1.5}}
\put(0.4,-1.6){\line(0,1){1}}
\put(1.4,-1.6){\line(0,1){2.1}}
\put(2.4,-1.6){\line(0,1){1}}
\put(-0.1,2.2){$\cdot$}
\put(0.9,2.2){$\cdot$}
\put(1.9,2.2){$\cdot$}
\put(2.9,2.2){$\cdot$}
\put(0.4,-2.2){$\cdot$}
\put(1.4,-2.2){$\cdot$}
\put(2.4,-2.2){$\cdot$}
\put(-0.1,3){$i_1$}
\put(0.9,3){$i_2$}
\put(1.9,3){$i_3$}
\put(2.9,3){$i_4$}
\put(0.4,-3){$j_1$}
\put(1.4,-3){$j_2$}
\put(2.4,-3){$j_3$}
\end{picture}\ \right\}.\ \text{ Then\ }\ 
\delta_{p}(i,j)=\left\{
    \begin{array}{ll}
        1 & \mbox{{if} } i_1=i_2=i_4=j_2 \\
        0 & \mbox{{otherwise.}}
    \end{array}
\right.
$$
\end{exe}

\begin{exe}\label{idtensor}
We give basic examples of such linear maps:
\begin{enumerate}
\item[(i)] T\_\setlength{\unitlength}{0.5cm}
$
\left\{\ \begin{picture}(0.01,1)\thicklines
\put(0,-0.6){\line(0,1){1.5}}
\end{picture}\ \right\}
=\emph{id}_{\mathbb{C}^N}$

\medskip
\item[(ii)] $T\_{\{\bigcap\}}(1)
=\sum_ae_a\otimes e_a$
\end{enumerate}
\end{exe}

Tensor products, compositions and involutions of diagrams behave as follows with respect to the associated linear maps:

\begin{prp}(\cite[Proposition 1.9]{BS09}\label{tensorcat}
Let $p,q$ be non-crossing partitions. Then:
\begin{enumerate}
\item $T_{p\otimes q}=T_p\otimes T_q$,
\item $T_{pq}=n^{-b(p,q)}T_pT_q$,
\item $T_{p^*}=T_p^*$.
\end{enumerate}
\end{prp}

We will keep the following notation in the sequel:

\begin{nota}\label{notanci}\
\begin{enumerate}
\item[$\bullet$] We will denote by $NC$ the collection of all sets $NC(k,l)$. It forms a monoidal category with involution and with $\N$ as a set of objects. 
\item[$\bullet$] We will denote by $NC^I$ the set of non-crossing partitions with upper and lower points decorated by elements of the set $I$. It is again a monoidal category whose objects are tuples of elements of $I$. The morphisms between a $k$-tuple and an $l$-tuple are the ones of $NC$.
\item[$\bullet$] We denote by $NC_s$ the subset of $NC^{\Z_s}$ defined as follows: 
\begin{enumerate}
\item[-] Objects are the same as the ones of $NC^{\Z_s}$ that is tuples of elements in $\Z_s$.
\item[-] Morphisms $p\in NC^{\Z_s}(\underline{i},\underline{j})$ are non-crossing partitions decorated by a $k$-tuple $\underline{i}=(i_1,\dots,i_k)\in\Z_s^k$ and an $l$-tuple $\underline{j}=(j_1,\dots,j_l)\in\Z_s^l$ having the property that, putting $\underline{i}$  on the upper row of $p$ and $\underline{j}$ on the lower row of $p$, then in each block, the sum over the elements of $\Z_s$ attached to the upper points is equal to the sum over the elements of $\Z_s$ attached to the lower points.
\end{enumerate}

\end{enumerate}
\end{nota}

For the term monoidal category, we refer to \cite[vii.1]{maclane}. The Proposition \ref{tensorcat} implies easily that the collection of spaces $\text{span}\{T_p : p\in NC(k,l)\}$ form a monoidal $C^*$-category in the sense of Doplicher and Roberts with $\N$ as a set of objects, see \cite{doplicher1989new} for the definitions. Furthermore, this tensor category has conjugates since the partitions of type

\setlength{\unitlength}{0.5cm}
$$ 
r=\left\{\ \begin{picture}(6,1.3)\thicklines

\put(0,0.5){\line(1,0){6}}
\put(2.7,1){$\emptyset$}
\put(0,-0.7){\line(0,1){1.22}}
\put(1,-0.7){\line(0,1){0.9}}
\put(5,-0.7){\line(0,1){0.9}}
\put(1,0.17){\line(1,0){4}}
\put(1.4,-0.7){...}
\put(3.6,-0.7){...}
\put(2.4,-0.7){\line(0,1){0.5}}
\put(3.4,-0.7){\line(0,1){0.5}}
\put(2.38,-0.21){\line(1,0){1.05}}
\put(6,-0.7){\line(0,1){1.22}}

\end{picture}\ \right\}\in NC(0;2k)$$
are non-crossing and since the following conjugate equations hold: 
\begin{equation}\label{eqconj}
(T_r^*\otimes \text{id})\circ(\text{id}\otimes T_r)=\text{id}=(\text{id}\otimes T_r^*)\circ(T_r\otimes \text{id}).
\end{equation}

Similar arguments show that the collection of spaces $\text{span}\{T_p : p\in NC_s(\underline{i},\underline{j})\}$ form a monoidal $C^*$-category with conjugates. 

In addition to Notation \ref{TENS} and \ref{notanci}, we will use the following one:
\begin{nota}\label{deflin}\
We will denote by $\emph{Lin}$ the ``projective" functor from $NC^I$ to $\emph{Vect}(H_i)$, where the $H_i$ are copies of $\mathbb{C}^N$ for some fixed N, defined as follows:
\begin{enumerate}
\item[$\bullet$] $\Lin({i_1},\dots,{i_k})=H_{i_1}\otimes\dots\otimes H_{i_k} : (i_1,\dots,i_k)\in I^k$,
\item[$\bullet$] $\Lin(p)=T_p\in B\left(\mathbb{C}^{N^{\otimes k}},\mathbb{C}^{N^{\otimes l}}\right) : p\in NC^I$.
\end{enumerate}
\end{nota}

\begin{rque}
The term ``projective" above is used because of the numerical factor appearing in the formula $T_{pq}=n^{-b(p,q)}T_pT_q$ above (Proposition \ref{tensorcat}). $\emph{Lin}$ is then a functor when one replaces the target vector spaces by the associated projective spaces where one quotients by the colinearity equivalence relation.
\end{rque}

Using Notation \ref{TENS}, \ref{notanci} and \ref{deflin}, we can now give a homogeneous result concerning the diagrammatic description of $\text{Tens}(S_N^+,v)$ and $\text{Tens}(H_N^{s+},U)$ (see \cite{Ban99}, \cite{BV09}) that we will generalize in the next subsections. All the categories considered below are then contained in $\text{Vect}(H_i)$ as defined above:

\begin{thm}\label{Tpindep}(\cite{Ban99}, \cite{BV09}, \cite{Tu92}) Let $N\ge2$, $s\in[1,\infty]$, $v$ be the fundamental corepresentation of $C(S_N^+)$ and $U$ be the fundamental corepresentation of $C(H_N^{s+})$ (see Section \ref{preliminaries}).
\begin{enumerate}
\item $\emph{Tens}(S_N^+,v )=\emph{span}\{ \Lin(NC)\}$ i.e. for all $k,l\in\N$ 
$$\emph{Hom}(v^{\otimes k},v^{\otimes l})={\emph{span}}\{T_p: p\in NC(k,l)\}.$$
\item $\emph{Tens}(H_N^{s+},\{U_i\})=\emph{span}\{\Lin(NC_s)\}$ i.e. for all $k,l\in\N$ and $i_t,j_q\in\Z_s$ (the case $s=\infty$ corresponds to $\Z$) $$\emph{Hom}(U_{i_1}\otimes\dots\otimes U_{i_k},U_{j_1}\otimes\dots\otimes U_{j_l})=\emph{span}\{T_p: p\in NC_s(\underline{i};\underline{j})\}.$$
\item\label{inddep} Moreover, the linear maps $T_p, p\in NC(k,l)$ are linearly independent for all $N\ge 4$.
\end{enumerate}
\end{thm} 

The proof of $(3)$ in the above theorem can be found in \cite{Tu92}. In this paper, Tutte computes the determinant of the Gram matrix associated with the vectors $T_p, p\in NC(0,k)$ and it is easy to see that it does not vanish for $N\ge4$.


We recall that in a monoidal $C^*$-category with conjugates, we have the following Frobenius reciprocity theorem (see \cite{Wor88} and \cite{Nesh}) that we will use in Proposition \ref{FRFP}.

\begin{thm}
Let $\mathscr{C}$ be a monoidal $C^*$-category with conjugates. If an object $U\in\mathscr{C}$ has a conjugate, with $R$ and $\overline{R}$ solving the conjugate equations (see \cite[Definition 2.2.1]{Nesh}, or (\ref{eqconj}) above), then the map
$$Mor(U\otimes V,W)\to Mor(V,\overline{U}\otimes W), T\mapsto (\emph{id}_{\overline{U}}\otimes T)(R\otimes \emph{id}_V)$$ is a linear isomorphism with inverse $S\mapsto (\overline{R}^*\otimes \emph{id}_W)(\emph{id}_U\otimes S)$.
\end{thm}

The next proposition is an application of Woronowicz's Tannaka-Krein duality. Tannaka-Krein duality was introduced by Woronowicz in \cite{Wor88}. It will be useful when computing the tensor category of certain compact quantum groups (see Theorem \ref{fusrules} below and compare with \cite[Theorem 12.1]{BBCC11}).

\begin{prp}\label{spanTK}
Let $\mathbb{G}_1=(C(\mathbb{G}_1),\{u_i\})$ and $\mathbb{G}_2=(C(\mathbb{G}_2),\{v_j\})$ be two compact quantum groups such that $C(\mathbb{G}_1), C(\mathbb{G}_2)$ are generated by the coefficients of some corepresentations $\{u_i\}, \{v_j\}$. 

Suppose that there is a surjective morphism $\pi: C(\mathbb{G}_1)\to C(\mathbb{G}_2)$ intertwining the coproducts (i.e. ${\mathbb{G}_2}\subset {\mathbb{G}_1}$). Suppose furthermore that $\ker(\pi)$ is generated by intertwining relations that is by a set $\mathscr{R}$ of linear maps $T$ which are morphisms in $\emph{Vect}(H_{u_i})$ giving equations in $C(\mathbb{G}_1)$. Then $\text{Tens}(\mathbb{G}_2,v_j)$ is generated as a rigid monoidal $C^*$-category by $\emph{Tens}(\mathbb{G}_1)$ and $\mathscr{R}$: $$\emph{Tens}(\mathbb{G}_2,\pi(u_i))=\langle \emph{Tens}(\mathbb{G}_1,u_i),\mathscr{R}\rangle.$$
\end{prp}
\begin{proof}
We denote by $\pi_*$ the functor associated to the homomorphism $\pi$. Let $\mathbb{G}'_2$ be the compact quantum group, obtained by Tannaka-Krein duality, whose representation category is the completion of $\langle \pi_*(\text{Tens}(\mathbb{G}_1,u_i)),\mathscr{R}\rangle$. By construction, the sets of the intertwining relations in $\mathbb{G}_2'$ and $\mathbb{G}_2$ coincide: they are composed of the relations in $\mathbb{G}_1$ and the additional ones described by $\mathscr{R}$. Thus, the morphism $\pi : C(\mathbb{G}_1)\to C(\mathbb{G}_2)$, with kernel $\mathscr{R}$, factorizes into an isomorphism of compact quantum groups $\pi' : C(\mathbb{G}_2')=C(\mathbb{G}_1)/\mathscr{R}\to C(\mathbb{G}_2)$. 
\end{proof}

\subsection{Intertwiner spaces of free product quantum groups}\label{fpfusrules}
In this section, we focus on compact \textit{matrix} quantum groups. We want a diagrammatic description of the spaces of intertwiners of the free products $$*_{i=1}^p(H_N^{s_i+}), s_i\in[1,+\infty).$$ We shall prove a more general result on the free product of a finite family of compact matrix quantum groups $\mathbb{G}=\star_{i=1}^p\mathbb{G}_i$. We will use the following notation: if $U_i$ is the fundamental corepresentation matrix of the compact quantum group $\mathbb{G}_i$ then $U_i^{\epsilon}$ will denote $U_i$ if $\epsilon=1$ and its conjugate $\overline{U_i}$ when $\epsilon=-1$.

\begin{prp}\label{FRFP}
Let $\mathbb{G}$ be a free product of a finite family of compact matrix quantum groups: $\mathbb{G}=\star_{i=1}^p(\mathbb{G}_i,U_i)$. We denote by $\mathscr{T}$ the category generated by the categories $\emph{Tens}(\mathbb{G}_i,U_i)$ that is
\begin{enumerate}
\item[$\bullet$] objects are tensor products $U_{r_1}^{\epsilon_{1}}\otimes\dots\otimes U_{r_k}^{\epsilon_{k}}$ ($r_i\in\{1,\dots,p\}$),
\item[$\bullet$] morphisms are linear combinations and compositions of morphisms of the type $\emph{id}\otimes R\otimes \emph{id}$ where $R$ is a morphism in a certain category $\emph{Tens}(\mathbb{G}_i,U_i)$.
\end{enumerate} 

Then we have $\emph{Tens}(\mathbb{G},\{U_i\}_{i=1}^p)=\mathscr{T}$.
\end{prp}

\begin{proof}
We first claim that $\mathscr{T}\subset \text{Hom}(\mathbb{G},\{U_i\})$. Indeed, if $$R\in \text{Hom}_{\mathbb{G}_r}(U_{r}^{\epsilon_1}\otimes\dots\otimes U_{r}^{\epsilon_k}, U_{r}^{\eta_1}\otimes\dots\otimes U_{r}^{\eta_l})$$ (for a fixed $r\in\{1,\dots,p\}$), then we clearly also have $$R\in \text{Hom}_{\mathbb{G}}(U_{r}^{\epsilon_1}\otimes\dots\otimes U_{r}^{\epsilon_k}, U_{r}^{\eta_1}\otimes\dots\otimes U_{r}^{\eta_l})$$ and moreover $\text{Tens}(\mathbb{G},\{U_i\}_{i=1}^p)$ is stable under the operations used to generate $\mathscr{T}$.


Now, we prove the other inclusion $\text{Tens}(\mathbb{G},\{U_i\}_{i=1}^p)\subset\mathscr{T}$.
The Frobenius reciprocity allows us to restrict to the cases $k=0$ : indeed the duality maps $T_r$ (see (\ref{eqconj})) are compositions of maps $\text{id}\otimes T_{r_i}\otimes \text{id}$ which are in $\mathscr{T}$. Denoting $1$ the trivial corepresentation, we have to prove that the morphisms $1\to U_{s_1}^{\eta_1}\otimes\dots\otimes U_{s_l}^{\eta_l} $ (i.e. the fixed vectors of this latter tensor product) are linear combinations of compositions of maps of the type $\text{id} \otimes R \otimes \text{id}$, where $R$ is a fixed vector for some factor $\mathbb{G}_i$.


Let $T$ be a morphism $T: 1\to U_{s_1}^{\eta_1}\otimes\dots\otimes U_{s_l}^{\eta_l} $ and set $V=U_{s_1}^{\eta_1}\otimes\dots\otimes U_{s_l}^{\eta_l} $. We will call any tensor product $$U_{s}^{\eta_t}\otimes\dots\otimes U_{s}^{\eta_{t+m}}=\bigotimes_{r=t}^{t+m}U_{s}^{\eta_r},$$ a sub-block of $V$ if this is a ``sub-tensor word" of $V$ such that $s_t=s_{t+1}=\dots=s_{t+m}=s$ ($1\le t\le\dots\le t+m\le l$) i.e. a tensor word in $U_s$ and $\overline{U_s}$ coming from the same copy $\mathbb{G}_i$ ; such a sub-block will be called \emph{maximal} when $s_{t-1}\ne s\ne s_{t+m+1}$ (whenever this is well defined).

We are going to prove the desired assertion by induction over the number of  maximal sub-blocks. The initialization corresponds to the case where there is only one maximal sub-block i.e. in this case each copy $U_{s_i}^{\eta_i}$ comes from the same factor $\mathbb{G}_{s_i}$ and thus the assertion is clear.

We denote the maximal sub-blocks by $$B_i=\bigotimes_{r=t_i}^{t_i+m_i}U_{s_i}^{\eta_{r}}, i=1,\dots, k$$ ($k\ge1$) with constant index $s_i$ and we set for all $i=1,\dots,k$ $$C_i=\underset{r>t_i+m_i}{\bigotimes_{r<t_i,}}U_{s_r}^{\eta_{r}}.$$ 
We fix $(S_{i}^j)_j$ an orthonormal basis of $\text{Fix}_{B_i}:=\text{Hom}_{\mathbb{G}_i}(1, B_i)$, note that $S_i^j\in \mathscr{T}$. 

Let $\forall i=1,\dots,k$, $P_i: B_i\to B_i$ be the orthogonal projections from the space of $B_i$ to $\text{Fix}_{B_i}$, that is, with our notation
$
P_i=\sum_jS_{i}^jS_{i}^{j*}.
$
Notice that $$(\text{id}\otimes P_i\otimes \text{id})\circ T=\sum_j(\text{id}\otimes S_{i}^j\otimes \text{id})\circ(\text{id}\otimes S_i^{j*}\otimes \text{id})\circ T,$$
and that for all $j$ $$[(\text{id}\otimes S_i^{j*}\otimes \text{id})\circ T: 1\to C_i]\in \text{Hom}_{\mathscr{T}}(1,C_i)$$ 
since $C_i$ has less maximal sub-blocks than $V$. So, we obtain that
\begin{equation}\label{recidT}
[(\text{id}\otimes P_i\otimes \text{id})\circ T : 1\to V]=\sum_j\lambda_j(\text{id}\otimes R_{j,1}\otimes \text{id})\circ\dots\circ R_{j,r}
\end{equation}
is a linear combination of composition of maps of type $\text{id}\otimes R_j\otimes \text{id}$ where $R_j$ is a fixed vector for some factor $\mathbb{G}_j,$ in other words $(\text{id}\otimes P_i\otimes \text{id})\circ T$ is a morphism in $\mathscr{T}$.

On the other hand, we can write  
\begin{equation}\label{expT1}
T=(P_1+(\text{id}-P_1))\otimes\dots\otimes( P_i+(\text{id}-P_i))\otimes\dots \otimes( P_k+(\text{id}-P_k))\circ T.
\end{equation}
We denote $\forall i=1,\dots,k$, 
$$
P_i^{(\epsilon_i)}= \left\{
    \begin{array}{ll}
        P_i & \mbox{if } \epsilon_i=1, \\
        \text{id}-P_i & \mbox{if } \epsilon_i=-1.\\
    \end{array}
\right.
$$

We are going to expand (\ref{expT1}) and conclude using the crucial result for our purpose, Theorem \ref{frprodcrucial}. The properties of the fusion rules of a free product of compact quantum groups recalled in Theorem \ref{frprodcrucial} yield that $P_1^{(-1)}\otimes \dots\otimes P_k^{(-1)}$ maps $V$ onto some direct sum $$\bigoplus_{r_i\ne1}r_1^{(s_1)}\otimes\dots\otimes r_k^{(s_k)}.$$ Indeed, each projection $P^{(-1)}_i$ maps the space of $B_i$ on the orthogonal complement of $\text{Fix}_{B_i}$ so that $P_i^{(-1)}(B_i)$ decomposes as a direct sum of \emph{non-trivial} irreducible $\mathbb{G}_i$-corepresentations. Hence, since the sub-blocks are maximal, the calculation rules of the alternated tensor words of corepresentations in a free product of compact quantum groups imply that the tensor product $P_1^{(-1)}(B_1)\otimes\dots\otimes P_k^{(-1)}(B_k)$ decomposes into non-trivial irreducible representations of the free product quantum group. 

Now, since $\text{Im}(T)$ is a copy of the trivial corepresentation $1\subset V$, we have $P_1^{(-1)}\otimes \dots\otimes P_k^{(-1)}\circ T=0$. Then we get
\begin{equation*}\label{T1}
T=\sum_{(\epsilon_1,\dots,\epsilon_k)\in \Omega}\left(P_1^{(\epsilon_1)}\otimes\dots\otimes P_k^{(\epsilon_k)}\right)\circ T,
\end{equation*}
where $\Omega=\{-1,+1\}^k\backslash\{(-1,-1,\dots,-1)\}$.

Hence we can write
\begin{align*}
T&=\sum_{(\epsilon_1,\dots,\epsilon_k)\in \Omega}\left(P_1^{(\epsilon_1)}\otimes\dots\otimes( P_{i_0})\otimes\dots\otimes P_k^{(\epsilon_k)}\right)\circ T &&\text{($\epsilon_{i_0}=1$)}\\ 
&=\sum_{(\epsilon_1,\dots,\epsilon_k)\in \Omega}\left(P_1^{(\epsilon_1)}\otimes\dots\otimes(\text{id})\otimes\dots\otimes P_k^{(\epsilon_k)}\right)\circ(\text{id}\otimes P_{i_0}\otimes \text{id})\circ T.
\end{align*}
and $T$ is as announced by (\ref{recidT}) and the definition of the maps $P_i$.

\end{proof}

\subsection{Intertwiner spaces in $H_N^+(\Gamma)$} 
Let $N$ be an integer $N\ge4$. Before proving a result for $H_N^+(\Gamma)=\widehat{\Gamma}\wr_* S_N^+$ similar to Theorem \ref{basicsAHS}, we give a collection of corepresentations, called basic corepresentations, for $H_N^+(\Gamma)=(C(H_N^+(\Gamma)),\Delta)$. We then describe their intertwiner spaces in terms of diagrams. With the notation of Example \ref{exefond}, we can obtain the following result.



\begin{prp}\label{basicGamma}
The algebra $C(H_N^+(\Gamma))$ has a family of $N$-dimensional (basic) unitary corepresentations $\{a(g): g\in\Gamma\}$ satisfying the conditions:
\begin{enumerate}
\item for all $g\in\Gamma$, $a(g)_{ij}=a_{ij}(g)$
\item $\overline{a(g)}=a(g^{-1})$.
\end{enumerate}
\end{prp}
\begin{proof}
It is clear, with the relations (\ref{tt1}), (\ref{tt2}), (\ref{tt3}), that setting $a(g)=(a_{ij}(g))_{i,j\in\{1,\dots,N\}}$ gives the desired family of corepresentations.
\end{proof}

\begin{nota}\label{notaLinbis} Let us first fix some notation. 
\begin{enumerate}
\item[$\bullet$] When $\mathscr{D}$ is a subset of diagrams in $NC^I$ for a certain set $I$, we denote by $\langle \mathscr{D}\rangle$ the set of all diagrams that can be obtained by usual tensor products, compositions and involutions of diagrams in $\mathscr{D}$ (see Definition \ref{NC}). 
\item[$\bullet$] If $\Gamma$ is finitely generated, we will only consider generating subsets $S_{\Gamma}$ of $\Gamma$ which are stable under inversion. We will consider the category $NC^{S_{\Gamma}}$ (see Notation \ref{notanci}).
\item[$\bullet$] We denote by $NC_{\Gamma,S_{\Gamma}}\subset NC^{S_{\Gamma}}$ the monoidal subcategory: 
\begin{enumerate}
\item[-] Where we keep all the objects of $NC^{S_{\Gamma}}$,
\item[-] Where the morphisms are the diagrams decorated by some elements $g_i\in S_{\Gamma}$ on upper points and by some elements $h_j\in S_{\Gamma}$ on lower points such that in each block $\prod_ig_i=\prod_jh_j$. We make the convention that the product of zero elements is equal to $e$.
\end{enumerate} 
\item[$\bullet$] We denote by $NC_{*}^{S_{\Gamma}}\subset NC^{S_{\Gamma}}$, the category: 
\begin{enumerate}
\item[-] Where objects are the same as in $NC^{S_{\Gamma}}$,
\item[-] Where morphisms are the ones in $NC_{\Gamma,S_{\Gamma}}$ which are such that for each block, there exists one element $g\in S_{\Gamma}$ such that the points of the block are decorated by $g$ or $g^{-1}$ (which might be equal).
\end{enumerate}
\item[$\bullet$] We denote by $NC_{\Gamma}\subset NC^{\Gamma}$, the category where the diagrams are decorated by elements $g_i,h_j\in \Gamma$ such that in each block $\prod_ig_i=\prod_jh_j$ with the same convention for products of zero elements as above. 
\end{enumerate}
\end{nota}

The following result will be useful in the proof of the next theorem. 
\begin{prp}\label{inftens} Let $\Gamma$ be a discrete group generated by a finite subset $S_{\Gamma}=\{g_1,\dots,g_p\}$ stable by inversion. We denote by $s_i$ the order of $g_i$ in $\Gamma$, with $s_i\in[1,+\infty]$. Then
$$\emph{Tens}\left(*_{i=1}^p(H_N^{s_i+}),\left\{U^{(i)}\right\}_{i=1}^p\right)=\emph{span}(\Lin(NC_*^{s_{\Gamma}})).$$



\end{prp}
\begin{proof}
We use Proposition \ref{FRFP} and Theorem \ref{Tpindep} $(2)$. We first prove the inclusion: $$\text{Tens}\left(*_{i=1}^p(H_N^{s_i+}),\left\{U^{(i)}\right\}_{i=1}^p\right)\subset\text{span}(\Lin(NC_*^{s_{\Gamma}})).$$ Notice that the morphisms of type $\id\otimes R\otimes\id$ with $R\in \text{Tens}\left(H_N^{s_i+},\left\{U^{(i)}\right\}\right)$ are linear combinations of diagrams of the form:

\setlength{\unitlength}{0.5cm}
$$\left\{\ \begin{picture}(6.2,2.6)\thicklines
\put(0,0.3){\line(0,1){1.3}}
\put(2.4,0.9){\line(0,1){0.7}}
\put(1.9,2){$g_{s_i}^{\pm1}$}
\put(1,0.3){\line(0,1){1.3}}
\put(0.8,2){$g_{s_2}$}
\put(0.1,-0.6){\dots}
\put(0.1,1.1){\dots}
\put(3.4,-1){\line(0,1){0.7}}
\put(3.2,-1.5){$g_{s_i}^{\pm1}$}
\put(3.1,2){$g_{s_i}^{\pm1}$}
\put(2.75,0.2){$p$}
\put(3.4,0.9){\line(0,1){0.7}}
\put(5.1,-0.6){$\dots$}
\put(5.1,1.1){$\dots$}
\put(5,-1){\line(0,1){2.6}}
\put(6,-1){\line(0,1){2.6}}
\put(4.8,2.){$g_{s_3}$}
\put(4.8,-1.5){$g_{s_3}$}
\put(5.8,2){$g_{s_4}$}
\put(5.8,-1.5){$g_{s_4}$}
\put(0,-1){\line(0,1){1.3}}
\put(-0.3,-1.5){$g_{s_1}$}
\put(1,-1){\line(0,1){1.3}}
\put(0.8,-1.5){$g_{s_2}$}
\put(2.4,-1){\line(0,1){0.7}}
\put(2,-0.3){\line(1,0){1.8}}
\put(2,0.9){\line(1,0){1.8}}
\put(2,-0.33){\line(0,1){1.26}}
\put(3.8,-0.33){\line(0,1){1.26}}
\put(2.1,-1.5){$g_{s_i}^{\pm1}$}
\put(-0.2,2){$g_{s_1}$}
\end{picture}\ \right\},
$$
for some $g_{s_j}\in S_{\Gamma}$, $p\in NC_{s_i}(g_{s_i}^{\pm1},\dots, g_{s_i}^{\pm1} ; g_{s_i}^{\pm1},\dots, g_{s_i}^{\pm1})$ where we identified $g_k^{\pm1}$ (of order $s_k\in[1,\infty]$) with $\pm1\in\Z_{s_k}$. It is clear that these diagrams are in $NC_*^{S_{\Gamma}}$.

The inclusion $\supset$ holds because any diagram in $NC_*^{S_{\Gamma}}$ decomposes as a vertical concatenation of diagrams of the above form. Indeed, since any partition $p\in NC_*^{S_{\Gamma}}$ is non-crossing, one can ``pull down" one block of $p$ without producing any crossing (such blocks are sometimes called internal blocks). The inclusion is then easily proved by induction over the number of blocks.
\end{proof}

We are now ready to prove the following theorem:

\begin{thm}\label{fusrulesbis} If $\Gamma$ is a finitely generated discrete group $\Gamma=\langle S_{\Gamma}\rangle$ with $S_{\Gamma}=\{g_1,\dots,g_p\}$ stable by inversion, then we have:
$$\emph{Tens} (H_N^+(\Gamma),a(g) : g\in S_{\Gamma})=\emph{span} \{\Lin\langle NC_{\Gamma,S_{\Gamma}}\rangle\}.$$
\end{thm}

\begin{proof}
We recall (see Proposition \ref{arrows}) that we have a morphism $\pi_1: *_{r=1}^pC(H_N^{s_r+})\to C(H_N^+(\Gamma))$ given by $$U_{i_1j_1}^{(r_1)}\dots U_{i_kj_k}^{(r_k)}\mapsto a_{i_1j_1}(g_{r_1})\dots a_{i_kj_k}(g_{r_k}),$$
with $p=|S_{\Gamma}|$.
We are going to determine the kernel of this morphism and apply Proposition \ref{spanTK}, Proposition \ref{FRFP} and Proposition \ref{inftens} to describe the category $$\text{Tens}\left(H_N^+(\Gamma), a(g) : g\in S_{\Gamma}\right).$$ 

\underline{Step 1:} We show that the kernel of $\pi_1$ is generated by the relations:
\begin{enumerate}
\item[-] $U_{ij}^{(r_1)}\dots U_{ij}^{(r_k)}=U_{ij}^{(s_1)}\dots U_{ij}^{(s_l)} \text{ if } \prod_i g_{r_i}=\prod_i g_{s_i} $,
\item[-] $U_{ij}^{(r)}U_{ik}^{(s)}=0=U_{ji}^{(r)}U_{ki}^{(s)} \text{ if } j\ne k$,
\item[-] $\sum_i U_{ij}^{(r_1)}\dots U_{ij}^{(r_k)}=\sum_j U_{ij}^{(r_1)}\dots U_{ij}^{(r_k)}=1$ \text{ if } $\prod_i g_{r_i}=e$.
\end{enumerate}
Indeed, if one denotes $I$ the associated ideal $I\subset *_{r=1}^pC(H_N^{s_r+})$, it is clear that $I\subset \ker(\pi_1)$. To prove the other inclusion, it is enough to prove that there is a morphism $s: C(H_N^+(\Gamma))\to *_{r=1}^pC(H_N^{s_r+})/I$ such that $s\circ \pi_1=q: *_{r=1}^pC(H_N^{s_r+})\to *_{r=1}^pC(H_N^{s_r+})/I$ is the canonical quotient morphism. We define $s$ as follows:
$$s(a_{ij}(g))=q(U_{ij}^{(r_1)}\dots U_{ij}^{(r_k)})$$ for $g=\prod_ig_{r_i}$. The relations satisfied by the elements $a_{ij}(g)$ (see Example \ref{exefond}) are also clearly satisfied by the elements $q(U_{ij}^{(r_1)}\dots U_{ij}^{(r_k)})$ in $*_{r=1}^pC(H_N^{s_r+})/I$. So by universality, $s$ is well defined.

Moreover, $s$ satisfies $s\circ \pi_1=q$. Hence, $\ker(\pi_1)=I$ is generated by the relations presented above.

\underline{Step 2:} 
We show that the relations generating $\ker(\pi_1)$ can be described by diagrams. With Notation (\ref{notaLinbis}), we claim that the one-block partitions $B_{k,l}\in NC_{\Gamma,S_{\Gamma}}, k,l\in\N$, decorated by certain elements $g_{r_1},\dots,g_{r_k}$ and $g_{s_1},\dots,g_{s_l}$ of $S_{\Gamma}$ with $\prod_i g_{r_i}=\prod_i g_{s_i} $ describe these relations. Their pictorial representations are as follows

\setlength{\unitlength}{0.5cm}
$$
B_{k,l}=\left\{\ \begin{picture}(5.4,2.6)\thicklines
\put(0,0.3){\line(0,1){1.3}}
\put(2.4,0.3){\line(0,1){1.3}}
\put(1,0.3){\line(0,1){1.3}}
\put(1.2,-0.5){\dots}
\put(1.2,1){\dots}
\put(0,0.3){\line(1,0){5.4}}
\put(3.4,0.3){\line(0,1){1.3}}
\put(4.4,0.3){\line(0,1){1.3}}
\put(5.4,0.3){\line(0,1){1.3}}
\put(4.9,2.1){$g_{r_k}$}
\put(0,-1){\line(0,1){1.3}}
\put(-0.3,-1.5){$g_{s_1}$}
\put(1,-1){\line(0,1){1.3}}
\put(2.4,-1){\line(0,1){1.3}}
\put(2.3,-1.5){$g_{s_l}$}
\put(-0.2,2.1){$g_{r_1}$}
\end{picture}\ \right\}\in NC_{\Gamma}((g_{r_1},\dots,g_{r_k});(g_{s_1},\dots,g_{s_l})).
$$
The conditions $$U_{ij}^{(r_1)}\dots U_{ij}^{(r_k)}=U_{ij}^{(s_1)}\dots U_{ij}^{(s_l)} \text{ with }\prod_i g_{r_i}=\prod_i g_{s_i} $$ follow from $$T_{k,l}:=T_{B_{k,l}}\in \text{Hom}\left(U^{(r_1)}\otimes\dots\otimes U^{(r_k)};U^{(s_1)}\otimes\dots\otimes U^{(s_l)}\right).$$ 
More precisely, if $(e_i)$ denotes the canonical basis of $\mathbb{C}^N$, we have 
\begin{align*}
&U^{(s_1)}\otimes\dots\otimes U^{(s_l)}(1\otimes T_{k,l})(1\otimes e_{p_1}\otimes\dots\otimes e_{p_k})=(1\otimes T_{k,l})U^{(r_1)}\otimes\dots\otimes U^{(r_k)}(1\otimes e_{p_1}\otimes\dots\otimes e_{p_k})\\
&\Leftrightarrow\delta_{p_1=\dots=p_k}{\sum_{i_1,\dots,i_l}}U^{(s_1)}_{i_1p_1}\dots U^{(s_l)}_{i_lp_1}\otimes e_{i_1}\otimes\dots\otimes e_{i_l}=(1\otimes T_{k,l}){\sum_{i_1,\dots,i_k}}U^{(r_1)}_{i_1p_1}\dots U^{(r_k)}_{i_kp_k}\otimes e_{i_1}\otimes\dots\otimes e_{i_k}\\
&\Leftrightarrow \delta_{p_1=\dots=p_k}{\sum_{i_1,\dots,i_l}}U^{(s_1)}_{i_1p_1}\dots U^{(s_l)}_{i_lp_1}\otimes e_{i_1}\otimes\dots\otimes e_{i_l}={\sum_{i_1}}U^{(r_1)}_{i_1p_1}\dots U^{(r_k)}_{i_1p_k}\otimes e_{i_1}^{\otimes l}.\\
\end{align*}
i.e. 
\begin{align*}
&T_{k,l}\in \text{Hom}\left(U^{(r_1)}\otimes\dots\otimes U^{(r_k)};U^{(s_1)}\otimes\dots\otimes U^{(s_l)}\right)\\
&\Leftrightarrow\left\{
    \begin{array}{ll}
        U^{(r_1)}_{ij}\dots U^{(r_k)}_{ij}=U_{ij}^{(s_1)}\dots U_{ij}^{(s_l)}\\\\
        U_{ip_1}^{(r_1)}\dots U_{ip_k}^{(r_k)}=0 \text{ if } p_t\ne p_s  \text{ and } U_{i_1p}^{(r_1)}\dots U_{i_kp}^{(r_k)}=0 \text{ if } i_t\ne i_s, \text{ for some } t,s.
    \end{array}
\right. 
\end{align*}
The last relations are equivalent to $U_{ij}U_{ik}=0=U_{ji}U_{ki}$ if $j\ne k$.
Similar computations with $l=0$ in the case $g_{r_1}\dots g_{r_k}=e$ give the relations $$\sum_i U_{ij}^{(r_1)}\dots U_{ij}^{(r_k)}=\sum_j U_{ij}^{(r_1)}\dots U_{ij}^{(r_k)}=1.$$

\underline{Step 3:} Notice that the one-block partitions $B_{k,l}$ generate with the usual tensor product, composition and involution operations all the diagrams in $NC_{\Gamma,S_{\Gamma}}$.
Thus, with Proposition \ref{spanTK} and Proposition \ref{inftens} and the fact that $NC_{*}^{S_{\Gamma}}\subset NC_{\Gamma,S_{\Gamma}} $ we get:
$$\text{Tens} (H_N^+(\Gamma),a(g) : g\in S_{\Gamma})=\text{span} \{\Lin\langle NC_{\Gamma,S_{\Gamma}},NC_{*}^{S_{\Gamma}}\rangle\}=\text{span}\{\Lin \langle NC_{\Gamma,S_{\Gamma}}\rangle\}.$$

\end{proof}
The previous theorem can be generalized as follows:

\begin{thm}\label{fusrules}
Let $\Gamma$ be any discrete group. Then for all $g_{1},\dots, g_{k}$ and $h_1,\dots,h_l$ elements in $\Gamma$:
\begin{align*}
\emph{Hom}_{H_N^+(\Gamma)}&\left(a(g_{1})\otimes\dots\otimes a(g_{k}),a(h_{1})\otimes\dots\otimes a(h_l)\right)
\\&=\emph{span}\left\{T_p: p\in NC_{\Gamma}((g_1,\dots,g_k);(h_1,\dots,h_l))\right\},
\end{align*}
where the sets $NC_{\Gamma}((g_1,\dots,g_k);(h_1,\dots,h_l))$ are composed of non-crossing partitions having the property that, when putting the elements $g_{i}$ and $h_{j}$ on the upper and lower row respectively, then in each block we must have $\prod_{i} g_{i}=\prod_{j} h_{j}$. 
\end{thm}

\begin{proof}
We consider the subgroup $\Gamma'\subset\Gamma$ generated by the elements $g_1,\dots,g_k$ and $h_1,\dots,h_l$. One can easily construct a generating subset $S_{\Gamma'}\subset\Gamma'$ of $\Gamma'$ containing the elements $g_1,\dots,g_k$ and $h_1,\dots,h_l$ which is stable by inversion. Notice that $C(H_N^+(\Gamma'))$ is a sub-Woronowicz-$C^*$-algebra of $C(H_N^+(\Gamma))$ and that the morphisms in $$\text{Hom}_{H_N^+(\Gamma)}\left(a(g_{1})\otimes\dots\otimes a(g_{k}),a(h_{1})\otimes\dots\otimes a(g_{s_l})\right)$$ are the ones in the full sub-category of intertwiners in $H_N^+(\Gamma')$. The result then follows from Theorem \ref{fusrulesbis} applied to the generating subset $S_{\Gamma'}$.
\end{proof}

We get the following corollary concerning the basic corepresentations:

\begin{crl}\label{crlbasics}
Let $N\ge4$. The basic corepresentations of $H_{N}^+(\Gamma)$ satisfy:
\begin{enumerate}
\item The corepresentations $a(g), g\ne e$ are irreducible.
\item $a(e)=1\oplus \omega(e)$, with $\omega(e)$ an irreducible corepresentation.
\item The corepresentations $\omega(e),a(g),g\in\Gamma\setminus\{e\}$ are pairwise non-equivalent.
\end{enumerate}
\end{crl}

\begin{proof}
We use the previous theorem and the fact that the linear maps $T_p$ are linearly independent (so that we can identify these maps with the associated non-crossing partitions). 
Let $g,h\in\Gamma$, the previous theorem gives that $$\text{dim}(\text{Hom}(a(g),a(h))=\# NC_{\Gamma}((g);(h)).$$
But it is easy to see that the only candidate elements in $NC_{\Gamma}((g);(h))$ are 
\setlength{\unitlength}{0.5cm}
$$
p=\left\{\ \begin{picture}(0.1,1.5)\thicklines
\put(0.05,-0.25){\line(0,1){0.8}}
\put(-0.2,0.9){$g$}
\put(-0.2,-0.9){$h$}
\end{picture}\ \right\} 
\text{ and } 
\setlength{\unitlength}{0.5cm}
q=\left\{\ \begin{picture}(0.1,1.5)\thicklines
\put(0,0.3){\line(0,1){0.5}}
\put(-0.2,1.1){$g$}
\put(0,-0.7){\line(0,1){0.5}}
\put(-0.2,-1.3){$h$}
\end{picture}\ \right\}$$ with the conditions $p\in NC_{\Gamma}((g);(h))\Leftrightarrow g=h$ and $q\in NC_{\Gamma}((g);(h))\Leftrightarrow g=h=e$.
Now we can compute the cardinal $\# NC_{\Gamma}((g);(h))$:
$$
\# NC_{\Gamma}((g);(h)) = \left\{
    \begin{array}{lll}
        0 & \mbox{if } g\ne h, \\
        1 & \mbox{if } g=h\ne e,\\
        2 & \mbox{if } g=h=e.
    \end{array}
\right.
$$
Then the second equality proves that the corepresentations $a(g), g\ne e$ are irreducible. The last equality, together with the fact that the trivial corepresentation is contained in $a(e)$ (since $\# NC_{\Gamma}((e);(e))=1$), prove that $a(e)=1\oplus \omega(e)$ where $\omega(e)$ is an irreducible corepresentation. And the fact that the basic corepresentations are pairwise non-equivalent comes from the first equality above.
\end{proof}

\subsection{Fusion rules for $H_N^+(\Gamma)$} In this subsection, $\Gamma$ denotes any discrete group and $N$ is an integer greater or equal $4$.
\begin{defi}
The fusion semiring $(R^+,\ ^-\ ,\oplus,\otimes)$ of $H_N^+(\Gamma)$ is defined as follows 
\begin{enumerate}
\item $R^+$ is the set of equivalence classes of corepresentations.
\item $(\ ^-\ ,\oplus,\otimes)$ are the usual involution, direct sum and tensor product of corepresentations of $H_N^+(\Gamma)$.
\end{enumerate}
\end{defi}

We now prepare the indexing of the irreducible corepresentations of $H_{N}^+(\Gamma)$.

\begin{defi}
Let $M=\langle\Gamma\rangle$ be the monoid formed by the words overs $\Gamma$. We endow $M$ with the following operations:
\begin{enumerate}
\item Involution: $(g_{1},\dots ,g_{k})^-=(g_{k}^{-1},\dots, g_{1}^{-1})$,
\item concatenation: for any two words, we set $$(g_{1},\dots,g_{k}),(h_{1},\dots,h_{l})=(g_{1},\dots,g_{{k-1}},g_{k},h_{1},h_{2},\dots,h_{l}),$$
\item Fusion: for two non-empty words, we set $$(g_{1},\dots,g_{k}).(h_{1},\dots,h_{l})=(g_{1},\dots,g_{{k-1}},g_{k}h_{1},h_{2},\dots,h_{l}).$$
\end{enumerate}
\end{defi}

\begin{nota}
If $(g_1,\dots,g_k)\in\langle\Gamma\rangle$, we write $|(g_1,\dots,g_k)|=k$ to denote the length of the word $(g_1,\dots,g_k)$.
\end{nota}

\begin{thm}\label{nonalter}
The irreducible corepresentations of $H_N^+(\Gamma)$ can be labelled by $\omega(x)$ with $x\in M$, with involution $\overline{\omega}(x)=\omega({\overline{x}})$ and the fusion rules:
$$\omega(x)\otimes\omega(y)=\sum_{x=u,t\ ;\ y=\overline{t},v}\omega({u,v})\ \oplus \displaystyle\sum_{\substack{x=u,t\ ;\ y=\overline{t},v\\ u\ne\emptyset,v\ne\emptyset}} \omega({u.v})$$ and we have for all $g\in\Gamma$, $\omega(g)=a(g)\ominus \delta_{g,e}1$.
\end{thm} 

For proving this theorem, we consider the set of irreducible corepresentations $\text{Irr}(H_N^+(\Gamma))$, the fusion semiring $R^+$ and the fusion ring $R$ of $H_N^+(\Gamma)$. They are contained in each other as follows: $$\text{Irr}(H_N^+(\Gamma))\subset R^+\subset R.$$
We also consider the additive monoid $\N\langle\Gamma\rangle$ with basis $B:=\{b_x: x\in\langle\Gamma\rangle\}$, involution $\overline{b}_x=b_{\overline{x}}$ and fusion rules 
\begin{equation}\label{fusrules1}
b_x\otimes b_y=\sum_{x=u,t\ ;\ y=\overline{t},v} b_{u,v}\ + \displaystyle\sum_{\substack{x=u,t\ ;\ y=\overline{t},v\\ u\ne\emptyset,v\ne\emptyset}} b_{u.v}.
\end{equation}
We want to prove that $\mathbb{N}\langle\Gamma\rangle\simeq R^+$ in such a way that $B$ corresponds to $\text{Irr}(H_N^+(\Gamma))$. To do this, we are going to construct an isomorphism of semi-rings $\Phi: \mathbb{N}\langle\Gamma\rangle\to R^+$. 

We construct and study $\Phi$ at the level of $\Z\langle\Gamma\rangle$ and $R$, where $\Z\langle\Gamma\rangle$ is the free $\Z$-module with basis $(b_x)_{x\in\langle\Gamma\rangle}$. 
It is clear that $(\Z\langle\Gamma\rangle,+,\times)$ is a free ring over $\Gamma$ for the product $b_x\times b_y=b_{x,y}$ where $x,y$ is the concatenation of the words $x$ and $y$. But this is not the ring structure we are interested in.

\begin{lem}
$\Z\langle\Gamma\rangle$ is also a free ring for the product defined by the fusion rules above (\ref{fusrules1}) and denoted $\otimes$ and with neutral element $b_{\emptyset}$ where $\emptyset$ denotes the empty word.
\end{lem}

\begin{proof}
Indeed, consider the ring $\mathbb{Z}\langle X_g, g\in\Gamma\rangle$ of non-commutative polynomials with variables indexed by $\Gamma$, and $$F: (\Z\langle X_g: g\in\Gamma\rangle,+,\times)\to(\Z\langle\Gamma\rangle,+,\otimes)$$ defined as a ring homomorphism by $X_g\mapsto b_g$. 
This morphism is bijective:

 For all $g_{1},\dots,g_{k}\in\Gamma$, we have by $(\ref{fusrules1})$:
\begin{equation}\label{calfus}
b_{g_{1},\dots ,g_{k}}=b_{{g_{1}},\dots ,g_{{k-1}}}\otimes b_{g_{k}}\ominus b_{g_1,\dots,g_{k-1}g_k} \ominus \delta_{(g_{{k-1}}g_{k},e)}b_{{g_{1}},\dots ,g_{{k-2}}}
\end{equation}
Then an induction over the length of the words $g_{1},\dots,g_{k}\in\langle\Gamma\rangle$ shows that $F$ is surjective.



Now we prove the injectivity of the morphism $F$. Let $P\in \Z\langle X_g: g\in\Gamma\rangle$ with $F(P)=0$. Suppose that $d:=\deg(P)\ge1$, i.e. that we can write $$P=\sum\lambda({g_{1},\dots,g_{d}})X_{g_{1}}\dots X_{g_{d}} +Q$$ with $\lambda(g_{1},\dots, g_{d})\ne0$, $\deg(Q)<\deg(P)$. 
Then, we have:
\begin{align*}
0=&F(P)=\sum\lambda(g_{1},\dots, g_{d})b_{g_{1}}\otimes\dots\otimes b_{g_{d}}+F(Q)\\
&=\sum\lambda(g_{1},\dots, g_{d})b_{g_{1},\dots, g_{d}}+c\ \ \ \ (c=\sum\mu_x b_x : |x|<d)\\
&\Rightarrow \lambda(x)=0\ \forall x \text{ with } |x|=d\ \text{(since $B$ is a basis of $\Z\langle\Gamma\rangle$)}.
\end{align*}
Thus we obtain $P=Q$ which contradicts $\deg(P)=d$.

The fact that $b_{\emptyset}$ is the multiplicative neutral element of $(\Z\langle\Gamma\rangle,+,\otimes)$ is clear by (\ref{fusrules1}).
\end{proof} 

Now, $(\Z\langle\Gamma\rangle,+,\otimes)$ being a free ring, we can define a ring homomorphism $\Phi: \Z\langle\Gamma\rangle\to R$ by $$\Phi(b_g)=\omega(g), g\in \Gamma,$$ where we put $\omega(g)=a(g)\ominus\delta_{g,e}1$ as in the statement. 

We shall prove that:
\begin{enumerate}
\item[$\bullet$] $\Phi: \Z\langle\Gamma\rangle\to R$ is injective,
\item[$\bullet$] $\Phi(B)\subset \text{Irr}(H_N^+(\Gamma))$,
\item[$\bullet$] $\Phi: B\to \text{Irr}(H_N^+(\Gamma))$ is surjective.
\end{enumerate}
The theorem, and in particular the fact that the conjugate $\overline{\omega}(x)$ of an irreducible equals $\omega(\bar x),$ will then follows easily.

We will denote $1$ both the trivial corepresentation in $H_N^+(\Gamma)$ and the neutral element $b_{\emptyset}$ in $\Z\langle\Gamma\rangle$ for $\otimes$. We will also denote by $z\mapsto\#(1\in z)$ both the linear forms counting:
\begin{enumerate}
\item[-] the number of copies of the trivial corepresentation contained in a $z\in R^+$,
\item[-] the coefficient of an element $z\in \Z\langle\Gamma\rangle$ relative to $b_{\emptyset}$.
\end{enumerate}
We will denote by $b_x^*(w)$ the coefficient of $w\in\Z\langle\Gamma\rangle$ relative to $b_x$ i.e. $\{b_x^* : x\in\langle\Gamma\rangle\}$ is the dual base of $\{b_x : x\in\Z\langle\Gamma\rangle\}$. That is, if $w=\sum\alpha_yb_y\in\Z\langle\Gamma\rangle$ we have $b_x^*(w)=\alpha_x\in\Z$. In particular, $\#(1\in z)=b_{\emptyset}^*(z)$. 

\begin{lem}\label{bigmap}
Let $k\in\N$. Let $g_1,\dots,g_k\in\Gamma$ and put $$P_k:=(b_{g_{1}}+ \delta_{g_{1},e}1)\otimes\dots\otimes (b_{g_{{k-1}}}+\delta_{g_{{k-1}},e}1)\otimes (b_{g_{k}}+ \delta_{g_{k},e}1).$$ Let $l\in\N$ and $\sigma : D_{\sigma}\to\{1,\dots,l\}$ be a surjective non-decreasing map defined on a subset $D_{\sigma}\subset\{1,\dots,k\}$. Put $x_j=\prod_{\sigma(i)=j}g_i$ and $x=(x_1,\dots,x_l)$. We allow $l$ to be $0$ and then in this case $x=\emptyset$.

Then $$b_x^*(P_k)=\#NC'_{\Gamma}((x_1,\dots,x_l) ; (g_1,\dots,g_k)),$$ where $NC'_{\Gamma}((x_1,\dots,x_l) ; (g_1,\dots,g_k))\subset NC_{\Gamma}((x_1,\dots,x_l) ; (g_1,\dots,g_k))$ is the sub-set composed of the elements $p\in NC_{\Gamma}((x_1,\dots,x_l) ; (g_1,\dots,g_k))$ with the additional rule that in each block there is at most one upper point and at least one lower point.
\end{lem}
\begin{proof}


We prove the lemma by induction over $k$. The case $k=1$ is easily proved. Indeed in this case, either $x=(g)$ for some $g\in\Gamma$, or $x=\emptyset$. We have $$b_{g}^*(b_{g}+\delta_{g,e}1)=1=\#NC_{\Gamma}'((g) ; (g))$$ and 

$$
b_{\emptyset}^*(b_{g}+\delta_{g,e}b_{\emptyset})=\left\{
    \begin{array}{ll}
        1 & \mbox{if } g=e \\
        0 & \mbox{otherwise}
    \end{array}
\right\}=\#NC'_{\Gamma}(\emptyset ; (g)).
$$

Let $k\ge2$ and assume that the result is proved for $k-1$. Let $g_k\in\Gamma$ and $x=(x_1,\dots,x_l)$ be a sequence as in the statement above. We consider $P_{k-1}\otimes (b_{g_k}+\delta_{g_k,e}1)$ and we first deal with the case $g_k\ne e$. We have 

\begin{equation}\label{analo}
b_{x}^*(P_{k-1}\otimes b_{g_k})=\delta_{g_k,x_l}b_{(x_1,\dots,x_{l-1})}^*(P_{k-1})+b_{(x_1,\dots,x_lg_k^{-1})}^*(P_{k-1})+b^*_{(x_1,\dots,x_l,g_k^{-1})}(P_{k-1}).
\end{equation}

The first term of (\ref{analo}) corresponds to the concatenation operation described by the fusion rules (\ref{fusrules1}) and, by induction, it is equal to $$\#NC'_{\Gamma}((x_1,\dots,x_{l-1}); (g_1,\dots,g_{k-1})).$$ The horizontal concatenation of such non-crossing partitions with the one-block partition 
\setlength{\unitlength}{0.5cm}
$
p=\left\{\ \begin{picture}(0.1,1.5)\thicklines
\put(0.05,-0.25){\line(0,1){0.8}}
\put(-0.2,0.9){$x_l$}
\put(-0.2,-0.8){$g_k$}
\end{picture}\ \right\} 
$ gives all non-crossing partitions in $NC'_{\Gamma}((x_1,\dots,x_l) ; (g_1,\dots,g_k))$ where $g_k$ is the only lower point in its block and connected to an upper point, $x_l=g_k$.

The second term of (\ref{analo}) corresponds to the fusion operation described in (\ref{fusrules1}) and, by induction, it is equal to $$\#NC'_{\Gamma}((x_1,\dots,x_{l-1},x_lg_k^{-1}); (g_1,\dots,g_{k-1})).$$ These non-crossing partitions carry the upper point $\{x_lg_k^{-1}\}$ and thus, because of the definition of $NC'_{\Gamma}$, we have $x_l=(\prod_ig_i)g_{k}$ for some $g_i\in\{g_1,\dots,g_{k-1}\}$. We obtain this way, all the non-crossing partitions in $NC'_{\Gamma}((x_1,\dots,x_l) ; (g_1,\dots,g_k))$ where $g_k$ is connected to some other lower points and to the upper point $\{x_l\}$.

The third and last term of (\ref{analo}), corresponds to the case where $g_k$ is the inverse of the last term in the sequence $(x_i)_i$, by induction, it is equal to $$\#NC_{\Gamma}'((x_1,\dots,x_l,g_k^{-1});(g_1,\dots,g_{k-1})).$$ These partitions carry the upper point $g_k^{-1}$ and thus we have $(\prod_ig_i)g_k=e$ for some $g_i\in\{g_1,\dots,g_{k-1}\}$. We obtain, this way, all the non-crossing partitions in $NC_{\Gamma}'((x_1,\dots,x_l);(g_1,\dots,g_k))$ where $g_k$ is connected to other lower points but to no upper point.

Altogether, we have proved $$b_x^*(P_{k-1}\otimes b_{g_k})=\#NC'_{\Gamma}((x_1,\dots,x_l) ; (g_1,\dots,g_k)).$$

In the case, $g_k=e$ we have
\begin{align*}
&b_{x}^*(P_{k-1}\otimes (b_{e}+1))\\
&=\left(\delta_{e,x_l}b_{(x_1,\dots,x_{l-1})}^*(P_{k-1})+b_{(x_1,\dots,x_l)}^*(P_{k-1})+b_{(x_1,\dots,x_l,e)}(P_{k-1})\right)+b_{(x_1,\dots,x_l)}^*(P_{k-1}),
\end{align*}
the additional term $b_{(x_1,\dots,x_l)}^*(P_{k-1})$ corresponding to the non-crossing partitions where $g_k=e$ is connected neither to another lower point, neither to an upper point. This case does not occur if $g_k\ne e$. The rest of the proof is similar to the other case.
\end{proof}

We get the following corollary:
\begin{lem}\label{cocom} For all $g_1,\dots,g_k\in\Gamma$, we have
\begin{align}\label{commute}
\#(1\in b_{g_{1}}\otimes\dots\otimes b_{g_{k}})=\#(1\in\omega({g_{1}})\otimes\dots\otimes\omega({g_{k}})).
\end{align}
\end{lem}
\begin{proof}
Notice that it is equivalent to show that: $\forall g_{1},\dots,g_{k}\in\Gamma$ 
\begin{align*}
\#\left(1\in \left[(b_{g_{1}}+ \delta_{g_{1},e}1)\otimes\dots\otimes (b_{g_{k}}+ \delta_{g_{k},e}1)\right]\right)=\#\left(1\in a(g_{1})\otimes\dots\otimes a({g_{k}})\right)\
\end{align*}
by definitions of $\omega({g})=a(g)\ominus \delta_{g,e}1$.
By Theorem \ref{fusrules} and the fact that the linear maps $T_p$ are linearly independent (see Theorem \ref{Tpindep}, these maps being the same maps as in $NC$), we have $$\#\left(1\in a(g_{1})\otimes\dots\otimes a({g_{k}})\right)=\#NC_{\Gamma}(\emptyset;(g_{1},\dots,g_{k})).$$ Then (\ref{commute}) follows from Lemma \ref{bigmap} with $x=\emptyset$.
\end{proof}


We can now prove the theorem:

\begin{proof}[Proof of Theorem \ref{nonalter}]\

\underline{Step 1:}
We first prove that $\Phi$ is injective. Let $\alpha\in \Z\langle\Gamma\rangle$ in the domain of $\Phi$. We denote by $\alpha^*$ the conjugate of $\alpha$ in $\Z\langle\Gamma\rangle$ (given on $B$ by $\bar b_x=b_{\bar x}$). Then, we have by Lemma \ref{cocom}
\begin{align*}
\Phi(\alpha)=0&\implies \Phi(\alpha\otimes\alpha^*)=0\\
&\implies \#(1\in\Phi(\alpha\otimes\alpha^*)=0)\\
&\implies\#(1\in\alpha\otimes\alpha^*)=0\ \text{by (\ref{commute})}\\
&\implies \alpha=0.
\end{align*}
The last implication comes from the fact that $\alpha\to\#(1\in\alpha\otimes\alpha^*)$ is a non-degenerate quadratic form, since for all words $w_1,w_2\in\langle\Gamma\rangle$, $$\#(1\in b_{w_1}\otimes \overline{b_{w_2}})=\#(1\in b_{w_1}\otimes b_{\overline{w_2}})=\delta_{\overline{w_2},w_1}.$$

\underline{Step 2:} Now we prove that $\Phi(B)\subset \text{Irr}(H_N^+(\Gamma))$ by induction on the length of the words $x\in\langle\Gamma\rangle$. It is clear that for all letters $g\in\Gamma\setminus\{e\}$, $\Phi(g)\in \text{Irr}(H_N^+(\Gamma))$ by Corollary \ref{crlbasics}. 

Now for a word of length $k$, $x=g_{i_1},\dots, g_{i_k}$. We have by (\ref{fusrules1}):
\begin{equation}
b_{g_{i_1},\dots ,g_{i_k}}=b_{g_{i_1}}\otimes b_{g_{i_2},\dots ,g_{i_k}}-b_{g_{i_1}g_{i_2},g_{i_3},\dots,g_{i_k}}-\delta_{g_{i_1}g_{i_2},e}b_{g_{i_3},\dots,g_{i_k}}\in\Z\langle\Gamma\rangle, 
\end{equation}
 and applying $\Phi$, we get 
$$\omega({g_{i_1},\dots ,g_{i_k}})=\omega({g_{i_1}})\otimes \omega({g_{i_2},\dots ,g_{i_k}})-\omega({g_{i_1}g_{i_2},g_{i_3},\dots,g_{i_k}})-\delta_{g_{i_1}g_{i_2},e}\omega({g_{i_3},\dots,g_{i_k}})\in R.$$
We want to prove that it is an element of $\text{Irr}(H_N^+(\Gamma))$. We first prove that it is an element of $R^+$ and then that it is an irreducible corepresentation. To fulfill the first part, since $\Phi$ is injective, we only have to prove that 
$$\text{Hom}\left(\omega({g_{i_1}g_{i_2},g_{i_3},\dots,g_{i_k}});\omega({g_{i_1}})\otimes \omega({g_{i_2},\dots ,g_{i_k}})\right)$$  
and
$$\text{Hom}(\omega({g_{i_3},\dots,g_{i_k}}) ; \omega({g_{i_1}})\otimes \omega({g_{i_2},\dots ,g_{i_k}}))$$ are one-dimensional.
We have by the Frobenius reciprocity 
\begin{align*}
&\text{dim}\ \text{Hom}(\omega({g_{i_1}g_{i_2},g_{i_3},\dots,g_{i_k}});\omega({g_{i_1}})\otimes \omega({g_{i_2},\dots ,g_{i_k}}))\\
&=\text{dim}\ \text{Hom}(1,\overline{\omega(g_{i_1}g_{i_2},g_{i_3},\dots,g_{i_k})}\otimes \omega({g_{i_1}})\otimes \omega({g_{i_2},\dots ,g_{i_k}}))\\
&=\#(1\in\overline{\omega(g_{i_1}g_{i_2},g_{i_3},\dots,g_{i_k})}\otimes \omega({g_{i_1}})\otimes \omega({g_{i_2},\dots ,g_{i_k}}))\\
&=\#(1\in\overline{b_{g_{i_1}g_{i_2},g_{i_3},\dots,g_{i_k}}}\otimes b_{g_{i_1}}\otimes b_{g_{i_2},\dots, g_{i_k}})\\
&=\#\left(1\in b_{g_{i_k}^{-1},g_{i_{k-1}}^{-1},\dots,g_{i_2}^{-1}g_{i_1}^{-1}}\otimes b_{g_{i_1}}\otimes b_{g_{i_2},\dots, g_{i_k}}\right)=1.
\end{align*}
The last equality comes from the facts that 
\begin{enumerate}
\item[$\bullet$] $1\in b_y\otimes b_x\Leftrightarrow y=\overline{x}$ (and in this case $\#(1\in b_y\otimes b_x)=1$),
\item[$\bullet$] $b_{g_{i_1}}\otimes b_{g_{i_2},\dots, g_{i_k}}=\sum\lambda_x b_x$ with $|x|=k-1, \lambda_x\ne0\Leftrightarrow x=(g_{i_1}g_{i_2},\dots,g_{i_k}).$
\end{enumerate}

A similar computation shows that $\text{Hom}(\omega({g_{i_3},\dots,g_{i_k}}) ; \omega({g_{i_1}})\otimes \omega({g_{i_2},\dots ,g_{i_k}}))$ is also one-dimensional in the case $g_{i_1}g_{i_2}=e$.

Finally, we can prove by similar arguments that $$\text{dim}\ \text{Hom}(\omega(g_{i_1},\dots, g_{i_k});\omega(g_{i_1},\dots, g_{i_k}))=1$$ i.e. $\omega(g_{i_1},\dots, g_{i_k})$ is irreducible.

\underline{Step 3:} We prove that $\Phi: B\to \text{Irr}(H_N^+(\Gamma))$ is surjective. An induction over $k$ shows that for all $g_1,\dots,g_k\in\Gamma$, we have:
$$b_{g_1}\otimes\dots\otimes b_{g_k}=\sum_l\sum_{j_1,\dots,j_l}C_{j_1\dots j_l}b_{g_{j_1},\dots, g_{j_l}}$$ for some coefficients $C_{j_1,\dots,j_l}\in\N$ so that, applying $\Phi$, $$\omega(g_1)\otimes\dots\otimes\omega(g_k)=\sum_l\sum_{j_1,\dots,j_l}C_{j_1\dots j_l}\Phi(b_{g_{j_1},\dots,g_{j_k}}).$$
Then any tensor product between basic corepresentations $\omega(g), g\in\Gamma$ is in $\Phi(\N B)=\text{span}_{\N}\langle B\rangle$, and the surjectivity follows since the coefficients of such tensor products generate $C(H_N^+(\Gamma))$ so that $\Phi(B)\supset \text{Irr}(H_N^+(\Gamma))$.

\underline{Step 4:} We can now conclude. The description of the irreducible corepresentations of $H_N^+(\Gamma)$ and the fusion rules between them then follows from the isomorphism of semi-rings $\N\langle\Gamma\rangle\simeq R^+$. 

Moreover, we have $\overline{\omega}(x)=\omega(\bar x)$ for all $x\in\langle\Gamma\rangle$. Indeed, since $\omega(x)$ is irreducible, $1\subset \omega(x)\otimes\omega(\bar x)$ if and only if $\overline{\omega}( x)=\omega(\bar x)$. But the fusion rules in $H_N^+(\Gamma)$ imply that we indeed have $1\subset\omega(x)\otimes\omega(\bar x)$. The proof of the theorem is then complete.
\end{proof}

We can give an alternative formulation of the description of these fusion rules: let $a$ be the generator of the monoid $\N$ with respect to the operation $+$ and $(z_g)_{g\in\Gamma}$ be a family of abstract elements satisfying exactly all the relations of the group $\Gamma$. We put $M'=\N*_e\Gamma$, the free product identifying both neutral elements of $\mathbb{N}$ and $\Gamma$ with the empty word. Then $M'$ is the monoid generated by the element $a$ and the family $(z_g)_{g\in\Gamma}$ with: 
\begin{enumerate}
\item[$\bullet$] involution: $a^*=a, z_g^*=z_{g^-1}$,
\item[$\bullet$] (fusion) operation inductively defined by: 
\begin{equation}\label{fusrules2}
vaz_g\otimes z_haw=vaz_{gh}aw+\delta_{gh,e}(v\otimes w),
\end{equation} 
\item[$\bullet$] unit $z_e$.
\end{enumerate}
Any element of $M'$ can be written as a ``reduced" word in the letters $a$ and $z_g, g\in\Gamma$, $\alpha=a^{l_1}z_{g_1}a^{l_2}z_{g_2}\dots a^{l_k}$ with 
\begin{enumerate}
\item[$\bullet$]  $l_1, l_k \ge 0$ and $l_i \ge 1$ for all $1<i<k$,
\item[$\bullet$] $g_i\neq e$ for all $i$ if $k> 1$,
\item[$\bullet$] $\alpha=a^l$ in the case $k=1$ for some $l\ge0$ and $a^0$ is the empty word equal to $z_e$.
\end{enumerate}

We obtain the following reformulation of the previous Theorem \ref{nonalter}:
\begin{thm}\label{alter}
The irreducible corepresentations $r_{\alpha}$ of $H_N^+(\Gamma)$ can be indexed by the elements $\alpha$ of the submonoid $S:=\langle az_ga: g\in\Gamma\rangle\subset M'$ and with fusion rules given by (\ref{fusrules2}). Furthermore, the basic corepresentations $\omega(g), g\in\Gamma\setminus\{e\}$ correspond to the words $az_ga$, $\omega(e)$ to $a^2$ and the trivial one to $a^0=1$.
\end{thm}

\begin{proof}
We first use the identification proved in the previous theorem: $\omega(g_{1},\dots, g_{k})\mapsto b_{g_{1}}\dots b_{g_{k}}$. Then, for any words $x,y\in\langle\Gamma\rangle$ and letters $g,h\in\Gamma$:
\begin{align*}
\omega(x,g)\otimes\omega(h,y)&=\underset{h,y=\overline{t},v}{\sum_{x,g=u,t}}\omega(u,v)\oplus\omega(u.v)\\
&=\omega(x,g,h,y)\oplus\omega(x,gh,y)\oplus\delta_{gh,e}\omega(x)\otimes\omega(y)
\end{align*}
Then with the identification mentioned above, we obtain this new (recursive) formulation for the fusion rules:
\begin{equation}\label{fusrules3}
pb_g\otimes b_hq=pb_gb_hq\oplus pb_{gh}q\oplus\delta_{gh,e}p\otimes q
\end{equation}
(with the identifications $\omega(x)\equiv p, \omega(y)\equiv q$). Now, we consider the submonoid $S\subset M'$ generated by elements $az_ga: g\in\Gamma$. It is a free monoid indexed by $\Gamma$ hence it is isomorphic to $\langle\Gamma\rangle$ (the monoid of the words over $\Gamma$ introduced in the previous theorem) via $b_g\equiv az_g a$. We have $(az_ga)^*=a{z_{g^{-1}}}a\in S$ so this identification is compatible with the involutions. 

Now let $p,q\in S$. We prove that with the fusion rules (\ref{fusrules2}) we can get back to (\ref{fusrules3}), and then the identification will preserve the fusion rules. It comes as follows:
\begin{align*}
pb_g\otimes b_hq &\equiv paz_ga\otimes az_haq\\ 
&=paz_ga^2z_haq\oplus paz_g\otimes z_haq\\
&\equiv pb_gb_hq\oplus pb_{gh}q\oplus \delta_{gh,e}p\otimes q.
\end{align*}

\end{proof}

\subsection{Dimension formula}
In this section we obtain the same dimension formula as in the case $\Gamma=\Z_s$ see \cite[Theorem 9.3]{BV09} and \cite[Corollary 2.2]{Lem13}. In this subsection $\Gamma$ is any discrete group, $N\ge2$.

Let us first fix some notation. Recall that there is a morphism $\pi : C(H_N^+(\Gamma))\to C(S_N^+)$ and that it corresponds to a functor $\pi : \text{Irr}(H_N^+(\Gamma))\to \text{Irr}(S_N^+)$, sending any $a(g), g\in\Gamma$ to the fundamental corepresentation $v$ of $S_N^+$. 

With the notation of Theorem \ref{nonalter} and Theorem \ref{alter} above, recall that if $r_{\alpha}\in \text{Irr}(H_N^+(\Gamma))$, we denote by $\chi_{\alpha}=(\text{id}\otimes \text{Tr})(r_{\alpha})$ the associated character. 

It is proved in \cite[Proposition 4.8]{Bra12} that the central algebra $C(S_N^+)_0=C^*-\langle\chi_k : k\in\N\rangle$ is isomorphic with $C([0,N])$ via $\chi_k\mapsto  A_{2k}(\sqrt{X})$ where $(A_k)_{k\in\N}$ is the family of dilated Tchebyshev polynomials defined inductively by $A_0=1, A_1=X$ and $A_1A_k=A_{k+1}+A_{k-1}$. 

We now give the following proposition whose proof can be found in \cite{Lem13} since the fusion rules between the irreducible corepresentations of $H_N^+(\Gamma)$ are similar for all groups $\Gamma$ (and $N\ge4$):

\begin{prp}(\/\cite[Proposition 2.1]{Lem13}\/)
Let $\chi_{\alpha}$ be the character of an irreducible corepresentation $r_{\alpha}\in \emph{Irr}(H_N^+(\Gamma))$. Write $\alpha=a^{l_1}z_{g_1}\dots a^{l_k}$. Then, identifying $C(S_N^+)_0$ with $C([0,N])$, the image of $\chi_{\alpha}$ by $\pi$, say $P_{\alpha}$, satisfies $$P_{\alpha}(X^2)=\pi(\chi_{\alpha})(X^2)=\prod_{i=1}^kA_{l_i}(X).$$
\end{prp}

\begin{crl}(\/\cite[Corollary 2.2]{Lem13}\/)
Let $r_{\alpha}$ be an irreducible corepresentation of $H_N^+(\Gamma)$ with $\alpha=a^{l_1}z_{g_1}\dots a^{l_k}$. Then $$\emph{dim}(r_{\alpha})=\prod_{i=1}^kA_{l_i}(\sqrt{N}).$$
\end{crl}

\section{Properties of the reduced operator algebra $C_r(H_N^+(\Gamma))$}

\subsection{Simplicity and uniqueness of the trace of $C_r(H_N^+(\Gamma))$}\label{simplicit}

In this subsection, we will assume $N\ge8$ and $|\Gamma|\ge2$. 
The case $|\Gamma|=1$ corresponds to $S_N^+$. The simplicity result proved in this subsection is already known (and we will use it) in the case of $S_N^+$ (see \cite{Bra12}). The statement of Theorem \ref{simpp} is then also true in the case $|\Gamma|=1$, but we emphasize that we only prove it in the case $|\Gamma|\ge2$ and that we use the case $|\Gamma|=1$.

The assumption on $N\ge8$, is due to the fact that the simplicity of $C_r(S_N^+)$ is only known in the cases $N\ge8$. For $N=2$, we know (see \cite{Bic04}) that $C_r(H_2^+(\Gamma))\simeq C_r^*(\Gamma*\Gamma)\otimes C(\Z_2)$ which is not simple. To summarize, the cases $3\le N\le7$ remain open. 

We denote by $||.||_r$ the norm on $C_r(H_N^+(\Gamma))$.

Let us fix some notation. We use the description of the irreducible corepresentations indexed by the monoid $M$ (see Theorem \ref{nonalter}) but we will simplify the notation $\omega(g_1,\dots,g_k)$ into $(g_1,\dots,g_k)$ and will denote the empty word by $1$ since it indexes the trivial representation. If $\alpha\in M$, we will denote the associated irreducible corepresentation by $r_{\alpha}$. We will denote by $|\alpha|=|(g_1,\dots,g_k)|$ the length $k\in\N$ of $\alpha\in M$. If $A\subset M$, we denote by $\overline{A}$ the set of conjugates $\bar\alpha$ of the elements $\alpha\in A$.

We will use the following notation as in \cite{MR1484551}. If $A,B\subset M$, we set $$A\circ B:=\{\gamma : \exists(\alpha,\beta)\in A\times B \text{ such that } r_{\gamma}\subset r_{\alpha}\otimes r_{\beta}\}\subset M.$$ 


We denote by $(g,\dots)\in M$ an element starting by $g\in\Gamma$ and $(\dots,g)$ an element ending by $g$. 


\begin{nota}\label{ensbles} We will denote $e$ the neutral element in $\Gamma$ and:
\begin{enumerate}
\item[$\bullet$] $e^k$ the word $(e,\dots,e)\in M$, with the convention $e^0=1$ and more generally $(g_1,\dots,g_i,e^k,g_{i+1},\dots,g_n)=(g_1,\dots,g_i,e,\dots,e,g_{i+1},\dots,g_n)$.
\item[$\bullet$] $E_1:=\bigcup\{(e,\dots)\}\cup\{1\}$ the subset of the words starting by $e$.
\item[$\bullet$] $E_2:=\bigcup_{k\in\N}\{e^k\}$ the subset of words with letters all equal to $e$ ($E_2\subset E_1$, note that we assume $0\in\N$).
\item[$\bullet$] $G_1:=\bigcup_{\underset{g\in\Gamma}{g\ne e}}\{(g,\dots)\}$ the subset of the words starting by any $g\ne e$ (notice that $M=E_1\sqcup G_1$).
\item[$\bullet$] $G_2:=\bigcup_{g,g'\ne e}\{(g,\dots,g')\}\subset G_1$.
\item[$\bullet$] $S:= {^cE_2}$ is the complement of $E_2$ in $M$.
\item[$\bullet$] $E_3:= S\cap E_1=E_1\setminus E_2$ (notice that $S=E_3\sqcup G_1$).
\end{enumerate}
\end{nota}

The definition of $ E_2$ and the fusion rules in Theorem \ref{nonalter}, clearly show that $1\in E_2$, $\overline{ E_2}= E_2$ and $ E_2\otimes E_2\subset E_2$. Let $\mathscr{C}'$ be the closure in $C_r(H_N^+(\Gamma))$ of the subspace $\mathscr{C}$ generated by the coefficients of the corepresentations $r_{\alpha}, \alpha\in E_2$. We know by \cite[Lemma 2.1, Proposition 2.2]{VerK}, that the Haar state $h_{\mathscr{C}'}$ is the restriction of the Haar state $h\in C_r(H_N^+(\Gamma))^*$ and that there exists a unique conditional expectation $P : C_r(H_N^+(\Gamma))\twoheadrightarrow \mathscr{C}'$ given by $h=h_{\mathscr{C}'}\circ P$. 

We recall from \cite{VerK}, that $P$ is defined by the compression by the orthogonal projection $p$ onto the closure of $\mathscr{C}'$ in $L^2(H_N^+(\Gamma))$: $P: C_r(H_N^+(\Gamma))\to p\mathscr{C}'p\simeq \mathscr{C}'$. 

We will denote by $\mathscr{S}'$ the closure of $\mathscr{S}:=\text{span}\{x\in \Pol(H_N^+(\Gamma)) : \supp(x)\subset S\}$ in $C_r(H_N^+(\Gamma))$. We have the algebraic decomposition as direct sum of vector subspaces $C_r(H_N^+(\Gamma))={\mathscr{C}'\oplus \mathscr{S}}'$, $P|_{\mathscr{C}'}=\text{id}$ and $\ker(P)=\mathscr{S}'$.

We are going to prove that $\mathscr{C}'$ can be identified with $C_r(S_N^+)$, using the simplicity of $C_r(S_N^+)$ (when $N\ge8$, see \cite{Bra12}) and adapt the ``modified Powers method" in \cite{MR1484551} where Banica proves the simplicity of $C_r(U_N^+)$.

\begin{prp}
$\mathscr{C}'\simeq C_r(S_N^+)$.
\end{prp}
\begin{proof}
We first notice that at the level of the universal $C^*$-algebras, we have $C(S_N^+)\simeq \overline{\mathscr{C}}^{||.||}$ where the closure is taken in $C(H_N^+(\Gamma))$. Indeed, with the notation of the first section (see Definition \ref{permw}, Example \ref{exefond}), we can construct by universal properties the following morphisms :
$$C(H_N^+(\Gamma))\overset{\pi_1}{\longrightarrow} C(S_N^+)\overset{\pi_2}{\longrightarrow} C(H_N^+(\Gamma)), \ a_{ij}(g)\mapsto v_{ij}\mapsto a_{ij}(e).$$

Now notice that $\forall x\in C(S_N^+)$, $\pi_1\circ\pi_2(x)=x$ and thus $C(S_N^+)$ is isomorphic with the (unital) sub-$C^*$-algebra of $C(H_N^+(\Gamma))$ generated by the elements $a_{ij}(e)$ i.e.
$$C(S_N^+)\simeq C^*-\langle x\in \Pol(H_N^+(\Gamma)) : \supp(x)\subset\{1,a(e)\}\rangle\subset C(H_N^+(\Gamma)).$$
But $\mathscr{A}:=*_{\text{alg}}-\langle x\in \Pol(H_N^+(\Gamma)) : \supp(x)\subset\{1,a(e)\}\rangle=\mathscr{C}$. Indeed : $a(e)=1\oplus r_{(e)}=1\oplus r_{e^1}$, thus the inclusion $\mathscr{A}\subset \mathscr{C}$ is clear. On the other hand, the coefficients of $1=r_{e^0}$ and $r_{e^1}=a(e)\ominus1$ are in $\mathscr{A}$. The inclusion $\mathscr{C}\subset \mathscr{A}$ then follows by induction since for all $k\in\N^*$, we have $$e^k=e^{k-1}\otimes e^1\ominus e^{k-1}.$$

Thus we obtain $\overline{\mathscr{C}}^{||.||}\simeq C(S_N^+)$ and then an isomorphism at the level of the reduced $C^*$-algebras (see e.g. \cite{VerK}) $\mathscr{C}'\simeq C_r(S_N^+).$
\end{proof}

Notice that we also proved $\mathscr{C}\simeq \Pol(S_N^+)$.

From now on all the closures are taken in the reduced $C^*$-algebra $C_r(H_N^+(\Gamma))$.
\medskip

Let $J \lhd C_r(H_N^+(\Gamma))$ be an ideal and let us prove that $J$ is either $\{0\}$ or $C_r(H_N^+(\Gamma))$. It is clear that $P(J)$ is an ideal in $\mathscr{C}'$. 
Hence, the simplicity of $\mathscr{C}'\simeq C_r(S_N^+)$ implies that $P(J)=\{0\}$ or $P(J)=\mathscr{C}'$.

Let us first assume that $P(J)=\{0\}$. Then since $\ker(P)=\mathscr{S}'$ we have $J\subset \mathscr{S}'$. 

But this is possible only if $J=\{0\}$. Indeed, since $1\notin S$ we have $\mathscr{S}'\subset \ker(h)$. But now for any $x\in J$, we have $x^*x\in J\subset \mathscr{S}'\subset \ker(h)$ and then $h(x^*x)=0$. Hence, $x=0$ since $h$ is faithful on $C_r(H_N^+(\Gamma))$. We then have $P(J)=\{0\}\Rightarrow J=\{0\}$.




In the sequel, we assume that $P(J)=\mathscr{C}'$ and we prove that $J=C_r(H_N^+(\Gamma))$. 

Since $P(J)=\mathscr{C}'\ni 1$, there exists $x\in J$ such that $x=1-z$ with $z\in \mathscr{S}'$. We write $z=z_0+(z-z_0)$ with $z_0\in \mathscr{S}$ and $||z-z_0||_r<1/2$. Notice that we can assume that $z$ and $z_0$ are hermitians, by taking the the real parts of $z$ and $z_0$ if necessary.

We are going to prove that we can find a finite family $(\beta_i)\subset \Pol(H_N^+(\Gamma))$ such that $C_r(H_N^+(\Gamma))\ni w\mapsto\sum_i\beta_iw\beta_i^*$ is unital and completely positive and $||\sum_{i}\beta_iz_0\beta_i^*||_r<1/2.$ We will then get 

\begin{align*}
||1-\sum_i\beta_ix\beta_i^*||_r&=||\sum_{i}\beta_i(1-x)\beta_i^*||_r=||\sum_i\beta_iz\beta_i^*||_r\\
&\le||\sum_i\beta_iz_0\beta_i^*||_r+||\sum_i\beta_i(z-z_0)\beta_i^*||_r\\
&\le||\sum_i\beta_iz_0\beta_i^*||_r+||z-z_0||_r<1,
\end{align*}
since any unital and completely positive map is contractive. Thus $\sum_i\beta_ix\beta_i^*\in J$ will be invertible and then we will get $J=C_r(H_N^+(\Gamma))$.

Let $g_0\in\Gamma\setminus\{e\}$ be an arbitrary chosen element different from $e$ (recall that $|\Gamma|\ge2$).
With the Notation \ref{ensbles}, we have the following proposition:
\begin{prp}\label{decomm}
Let $\alpha_1:=(g_0,e)$, $\alpha_3:=(g_0,e^3)$, $\alpha_5:=(g_0,e^5)$ in $S$. Let $G\subset S$ finite. Then:
\begin{enumerate}
\item\label{decomp} $S=E_3\sqcup G_1, G_2\circ E_1\cap E_1=\emptyset, \{\alpha_t\}\circ G_1\cap \{\alpha_s\}\circ G_1=\emptyset,\forall t\ne s$ in $\{1,3,5\}$,
\item\label{gs} $\bigcup_{t\in\{1,3,5\}}\{\alpha_t\}\circ G_2\circ \{\overline{\alpha_t}\}\subset G_2$,
\item\label{perm} $\exists\alpha\in S$ s.t. $\left\{\alpha\right\}\circ G\circ \left\{\overline{\alpha}\right\}\subset G_2$.
\end{enumerate}
\end{prp}

\begin{proof}
The first assertion in $(1)$ is clear. The second follows from the following computations for $g,g'\ne e$:
$$(g,\dots,g')\otimes(e,\dots)=(g,\dots,g',e,\dots)\oplus(g,\dots,g',\dots),$$ and
 the third from 
$$(g_0,e^t)\otimes(g,\dots)=(g_0,e^t,g,\dots)\oplus(g_0,e^{t-1},g,\dots)$$ for all $g\ne e$ and $t=1,3,5$ (where we made the convention that in the case $t=1$, the letter $1=e^0$ is deleted in the second term of the right hand side in the above equality).

The assertion $(2)$ follows from the following computation for $g,g'\ne e$:
\begin{align*}
(g_0,e^t)\otimes(g,\dots,g')&\otimes(e^t,g_0^{-1})=((g_0,e^t,g,\dots,g')\oplus(g_0,e^{t-1},g,\dots,g'))\otimes(e^t,g_0^{-1})\\
&=(g_0,e^t,g,\dots,g',e^t,g_0^{-1})\oplus(g_0,e^t,g,\dots,g',e^{t-1},g_0^{-1})\oplus\\
&\oplus(g_0,e^{t-1},g,\dots,g',e^t,g_0^{-1})\oplus(g_0,e^{t-1},g,\dots,g',e^{t-1},g_0^{-1}).
\end{align*}

For $(3)$ consider $\alpha=(g_0,e,\dots,e)=(g_0,e^m)$ where the number of letters $e$ in the word $\alpha$ is greater than $m:=\max\{|\beta| : \beta\in G\}$. For any $\gamma\in G$, we can write $\gamma=(e^{l-1},h_l,\dots,h_k)$ with $h_l\ne e, 1\le l\le k$. Then 
\begin{align*}
\alpha\otimes\gamma=\bigoplus_{s=-(l-1)}^{l-1}(g_0,e^{m+s},h_l,\dots,h_k)\ \text{ since } l-1<k\le m.
\end{align*}
Now write $\gamma=(e^{l-1},h_l,\dots,h_{l'},e^{k-l'})$ with $l\le l'\le k$ and compute
\begin{align*}\alpha\otimes\gamma\otimes\bar\alpha&=(g_0,e^m)\otimes\gamma\otimes(e^m,g_0^{-1})\\
&=\bigoplus_{r=-(k-l')}^{k-l'}\bigoplus_{s=-(l-1)}^{l-1}(g_0,e^{m+s},h_l,\dots,h_{l'},e^{r+m},g_0^{-1}).\\
\end{align*}
Hence, $\{(g_0,e^m)\}\circ G\circ\{(e^m,g_0^{-1})\}\subset G_2$.
\end{proof} 

We shall apply Proposition \ref{decomm} to $G=\supp(z_0)$ (with $z_0\in \mathscr{S}$ as above Proposition \ref{decomm}). We will then get an element $z'=\sum_ia_iz_0a_i^*$ with support in $G_2$, where $(a_i)\subset \Pol(H_N^+(\Gamma))$ is a finite family of coefficients of $r_{\alpha}$ with $\alpha\in S$ obtained by the assertion $(3)$ of the previous proposition. 
The proof of the simplicity of $C_r(H_N^+(\Gamma))$ will then rely on a result proved in \cite{MR1484551} that we recall below.




We recalled in Section \ref{preliminaries} the construction of the adjoint representation of a compact quantum group of Kac type. We obtained for all irreducible characters $\chi_r$, a completely positive map $\text{ad}(\chi_r)\in B(C_r(H_N^+(\Gamma))),$ such that $$\text{ad}(\chi_r)(z)=\sum_{i,j}r_{ij}zr_{ij}^*.$$ We will simply denote $\text{ad}(r):=\text{ad}(\chi_r)$, if $r\in \text{Irr}(H_N^+(\Gamma))$. We easily see that $\text{ad}(r)(1)=\dim(r)1$ and $\tau(\text{ad}(r)(z))=\dim(r)\tau(z)$ for any trace $\tau\in C_r(H_N^+(\Gamma))^*$. With these notation, we will take $z'=\text{ad}(r_{\alpha})(z_0)$.

We put $d_t=\dim(r_{\alpha_t})$ with $\alpha_t, t=1,3,5$ defined in Proposition \ref{decomm}. Then, if one considers the maps $\dfrac{\text{ad}(r_{\alpha_t})}{d_t}$, one can get as in \cite[Proposition 8]{MR1484551}:

\begin{prp}\label{construct}
The unital and completely positive linear map $T : C_r(H_N^+(\Gamma))\to C_r(H_N^+(\Gamma)),$ $T=\sum_{t\in\{1,3,5\}}\dfrac{\text{ad}(\alpha_t)}{3d_t}$ is such that: 
\begin{enumerate}
\item\label{mapp} $T(z)=\sum_i a_iza_i^*$ for some finite family $(a_i)\subset \Pol(H_N^+(\Gamma))$.
\item $T$ is $\tau$-preserving for any trace $\tau\in C_r(H_N^+(\Gamma))^*$ (hence $h$-preserving).
\item For all $z=z^*\in C_r(H_N^+(\Gamma))$ with $ \supp(z)\circ E_1\cap E_1=\emptyset$, we have $||T(z)||_r\le0.95||z||_r$ and $\supp(T(z))\subset\bigcup_t\{\alpha_t\}\circ \supp(z)\circ\{\overline{\alpha_t}\}$.
\end{enumerate}
\end{prp}

Then, we can get the simplicity of $C_r(H_N^+(\Gamma))$.

\begin{thm}\label{simpp}
$C_r(H_N^+(\Gamma))$ is simple with unique trace $h$, for all $N\ge8$ and any discrete group $\Gamma$, $|\Gamma|\ge2$.
\end{thm}
\begin{proof}
We denote by $\tau$ any faithful normal trace on $C_r(H_N^+(\Gamma))$. 

We recall that we assume now $P(J)=\mathscr{C}'$ (we already proved above that $P(J)=\{0\}\Rightarrow J=\{0\}$). 
By the discussion before Proposition \ref{decomm}, it remains to prove that there exists a finite family $(\beta_i)\subset \Pol(H_N^+(\Gamma))$ such that $||\sum_i\beta_iz_0\beta_i^*||_r<1/2$.

We first apply Proposition \ref{construct} to the hermitian $z'=\text{ad}(r_{\alpha})(z_0)$ with $\alpha$ obtained in Proposition \ref{decomm}. Notice that $\text{ad}(r_{\alpha})(z_0)=\sum_ia_iz_0a_i^*\in\mathscr{S}$. 
We then get a unital $\tau$-preserving linear map $V_1 : C_r(H_N^+(\Gamma))\to C_r(H_N^+(\Gamma))$ of the form $z\mapsto \sum_ic_izc_i^*$ where $(c_i)\subset \Pol(H_N^+(\Gamma))$ is a finite family, with 
$$||V_1(z')||_r\le0.95||z'||_r$$
since $\supp(z')\subset G_2$ and $G_2\circ E_1\cap E_1=\emptyset$. Moreover,
$$\supp(V_1(z'))\subset\cup_t\{\alpha_t\}\circ \supp(z')\circ \{\overline{\alpha_t}\}\subset G_2.$$
Then the $\tau$-preserving map $z\mapsto V_1\text{ad}(r_{\alpha})(z)$ is of the form $z\mapsto\sum_id_izd_i^*$ for some finite family of elements $d_i\in \Pol(H_N^+(\Gamma))$, and satisfies $$V_1(\text{ad}(r_{\alpha})(z_0))^*=V_1(\text{ad}(r_{\alpha})(z_0)).$$

Thus we can apply the same arguments we just used as many times as needed so that there exists maps $V_2,\dots, V_m$ similar to $V_1$, such that $$||V_m\dots V_1\text{ad}(r_{\alpha})(z_0)||_r<1/2.$$ We put $V:=V_m\dots V_1\text{ad}(r_{\alpha})$, which is a completely positive, unital map in $B(C_r(H_N^+(\Gamma)))$ of the form $$z\mapsto \sum_i\beta_iz\beta_i^*,$$  where $(\beta_i)\subset \Pol(H_N^+(\Gamma))$ is finite and $||\sum_i\beta_iz_0\beta_i^*||_r<1/2$. 


The simplicity of $C_r(H_N^+(\Gamma))$ then follows from these observations and the discussion before Proposition \ref{decomm}.

For the uniqueness of the trace let us first show that if $\tau$ is any (positive) trace on $C_r(H_N^+(\Gamma))$ then the restriction to $\mathscr{S}'$ is the restriction of the Haar state $h$. Indeed, let $x=x^*\in \mathscr{S}$. Then, on the one hand we have $h(x)=0$ and on the other hand if $\epsilon>0$, we can apply the method above to find a finite family $(a_i)\subset \Pol(H_N^+(\Gamma))$ such that $y=\sum_ia_ixa_i^*$ has norm less than $\epsilon>0$ and thus $|\tau(x)|=|\tau(y)|\le\epsilon$. Letting $\epsilon\to0$, we get that $h$ and $\tau$ coincide on the hermitians of $\mathscr{S}$. But any $x\in \mathscr{S}$ is the sum of two hermitians hence $\tau|_{\mathscr{S}}=h|_{\mathscr{S}}$. By continuity, we get $\tau|_{\mathscr{{S'}}}=h|_{\mathscr{S'}}$.

Now take any $x\in C_r(H_N^+(\Gamma))$ and write $x=y+z$ with $y\in \mathscr{C'}$ and $z\in \mathscr{S'}$. One has $$\tau(y+z)=\tau(y)+\tau(z)=h_{S_N^+}(y)+0$$
by the uniqueness of the trace on $C_r(S_N^+)$ and what we just proved on $\mathscr{S}'$. We then get $\tau(x)=h_{S_N^+}(y)=h(x)$ hence $\tau=h$.

\end{proof}
In particular there is a unique faithful normal trace on the von Neumann algebra $L^{\infty}(H_N^+(\Gamma))$, given by the extension of the Haar state $h$ i.e.:
\begin{crl}
$L^{\infty}(H_N^+(\Gamma))$ is a $II_1$-factor for all $N\ge8$, and any discrete group $|\Gamma|\ge2$.
\end{crl}

\subsection{Fullness of the $II_1$-factor $L^{\infty}(H_N^+(\Gamma))$}\

In this subsection, $N$ is again an integer greater than $8$ and $\Gamma$ is any discrete group $|\Gamma|\ge2$. Once more the case $3\le N\le7$ remain open. If $N=2$, we have that $L^{\infty}(H_2^+(\Gamma))=L(\Gamma*\Gamma\times\Z_2)$ and it is not a factor since $\Gamma*\Gamma\times\Z_2$, containing $\Z_2$ in its center, is not icc . 
We will denote by $||.||_2$ the $L^2(H_N^+(\Gamma))$-norm with respect to the tracial Haar state $h$. 

\begin{defi}(\cite{MvN43})
Let $(M,\tau)$ be a $II_1$-factor with unique faithful normal trace $\tau$. A sequence $(x_n)\subset M$ is said to be asymptotically central if for all $y\in M$, $||x_ny-yx_n||_2\to0$. We say that $(x_n)$ is asymptotically trivial if $||x_n-\tau(x_n)1||_2\to0$. The $II_1$-factor $(M,\tau)$ is said to be full if every bounded asymptotically central sequence is trivial. 
\end{defi}

We use the decomposition of the preceding subsection \ref{simplicit}, $\Pol(H_N^+(\Gamma))=\mathscr{C}\oplus \mathscr{S}$ which gives the vector space decomposition $L^{\infty}(H_N^+(\Gamma))=\overline{\mathscr{C}}^{\sigma-w}\oplus \overline{\mathscr{S}}^{\sigma-w}$. 

We want to prove the fullness of $L^{\infty}(H_N^+(\Gamma))$. It is easy to see that it is enough to consider sequences in the dense subalgebra $\Pol(H_N^+(\Gamma))$.
We then fix a bounded asymptotically central sequence $(x_n)\subset \Pol(H_N^+(\Gamma)$. One can write $x_n=y_n+z_n$, with $y_n\in \mathscr{C}, z_n\in \mathscr{S}$. If $a\in \overline{\mathscr{C}}^{\sigma-w}$, we have
\begin{align*}
||y_na-ay_n||_{2,L^2(\mathscr{{C'}})}&=||y_na-ay_n||_{2}\\
&=||P(x_n)a-aP(x_n)||_{2}\\
&=||P(x_na)-P(ax_n)||_{2}\le||x_na-ax_n||_2
\end{align*}
where the last $2$-norms are the $L^2(H_N^+(\Gamma))$-norm. The first equality above comes from the fact that the restriction to $\mathscr{C}'$ of the Haar state $h$ of $C_r(H_N^+(\Gamma))$ is the Haar state on $\mathscr{{C'}}$.
We then get $||y_na-ay_n||_{2}\to0$, $\forall a\in \overline{\mathscr{C}}^{\sigma-w}\simeq L^{\infty}(S_N^+)$. As a result, we obtain that $||y_n-h(y_n)1||_{2}\to0$ since $L^{\infty}(S_N^+)$ is a full factor for $N\ge8$ (see \cite{Bra12}). In particular, $(y_n)$ is asymptotically central in $L^{\infty}(H_N^+(\Gamma))$, and hence this is also the case of $(z_n)$. 

To get the fulness of $L^{\infty}(H_N^+(\Gamma))$ it remains to prove that $(z_n)$ is asymptotically trivial. To do this we adapt the $``14-\epsilon$ method" introduced, in particular, to prove that $L(F_n)$ does not have the property $\Gamma$, see \cite{MvN43}. Our proof is based on the arguments used by Vaes to prove the fulness of $L^{\infty}(U_N^+)$, see the appendix in \cite{CFY13}.

Once more, we use the decomposition $S=E_3\sqcup G_1$ which gives two orthogonal subspaces in $L^2(H_N^+(\Gamma))$:
$$H_1=\overline{\text{span}}^{||.||_2}\{\Lambda_h(x) : \supp(x)\subset E_3\},$$
$$H_2=\overline{\text{span}}^{||.||_2}\{\Lambda_h(x) : \supp(x)\subset G_1\}$$
where $\Lambda_h$ is the GNS map associated to the Haar state $h$. Let us set $H:=H_1\perp H_2$, and $H_1^0, H_2^0$ the corresponding subspaces before taking closures. If $x\in L^{\infty}(H_N^+(\Gamma))$, we will simply write $x\equiv\Lambda_h(x)$ via $L^{\infty}(H_N^+(\Gamma))\hookrightarrow L^2(H_N^+(\Gamma))$. 

If $\beta\in S$, we set $d_{\beta}:=\dim(r_{\beta})$ and $K_{\beta}:=L^2(B(H_{\beta}),\frac{1}{d_{\beta}}Tr(\cdot))$, where $H_{\beta}$ is the representation space of $r_{\beta}\in \mathscr{S}\otimes B(H_{\beta})$, and we consider the isometry:
$$v_{\beta} : H\to H\otimes K_{\beta}, v_{\beta}a=r_{\beta}(a\otimes 1)r_{\beta}^*=r_{\beta}(a\otimes 1)r_{\overline{\beta}}.$$
We will denote the norm and scalar product on $L^2(H_N^+(\Gamma))$ by $||.||_2, \langle\cdot,\cdot\rangle_2$ and the scalar product and norm on the tensor spaces $H\otimes K_{\beta}$ simply by $||.||, \langle\cdot,\cdot\rangle$.


We fix an element $g\in\Gamma$, $g\ne e$. Remark that $g\ne e$ implies that we must assume $|\Gamma|\ge2$. The case $|\Gamma|=1$ corresponds to $S_N^+$ and the result we want to prove is already known in this case. 

We recall that we denote by $e^k$, the word $(e,\dots,e)$ with $k$ letters equal to $e$.

\begin{lem}\label{bidule} With the notation above, we have
\begin{enumerate}
\item $\forall \beta\in E_3$, $\{(g)\}\circ\beta\circ\left\{\overline{(g)}\right\}\subset G_1$,
\item $\{e^i\}\circ G_1\circ \{e^i\}\subset E_3, i=2,4$,
\item $\{e^2\}\circ G_1\circ \{e^2\}\cap \{e^4\}\circ G_1\circ \{e^4\}=\emptyset$.
\end{enumerate}
\end{lem}

\begin{proof}
Let $\beta\in E_3$, i.e. $\beta=(e,h_1,\dots,h_l)$ with $l\ge1$ and $h_i\ne e$ at least for one $i\in\{1,\dots,l\}$, we have: 
\begin{align*}
(g)\otimes\beta\otimes\overline{(g)}&=(g,e,h_1,\dots,h_l)\otimes\overline{(g)}\oplus(g,h_1,\dots,h_l)\otimes\overline{(g)}\\
&=(g,e,h_1,\dots,h_l,g^{-1})\oplus(g,e,h_1,\dots,h_lg^{-1})+\delta_{h_lg^{-1},e}(g,e,h_1,\dots,h_{l-1})\\
&\oplus(g,h_1,\dots,h_l,g^{-1})\oplus(g,h_1,\dots,h_lg^{-1})+\delta_{h_lg^{-1},e}(g,h_1,\dots,h_{l-1})
\end{align*}
and then $(1)$ follows.

Now let $\alpha\in G_1$, i.e. $\alpha=(h_1,\dots,h_l)$ with $h_1\ne e$. We have 
\begin{align*}
(e^i)\otimes\alpha\otimes(e^i)&=((e^i,h_1,\dots,h_l)\oplus(e^{i-1},h_1,\dots,h_l))\otimes(e^i)\\
&=(e^i,h_1,\dots,h_l,e^i)\oplus(e^i,h_1,\dots,h_l,e^{i-1})+\delta_{h_l,e}(e^i,h_1,\dots,h_{l-1})\otimes(e^{i-2})\\
&\oplus(e^{i-1},h_1,\dots,h_l,e^i)\oplus(e^{i-1},h_1,\dots,h_l,e^{i-1})+\delta_{h_l,e}(e^{i-1},h_1,\dots,h_{l-1})\otimes(e^{i-2})
\end{align*}
and this gives $(2), (3)$.
\end{proof}

\begin{prp}
For all $z\in \mathscr{S}$, we have $$||z||_2\le14\max\{||z\otimes1-v_{(g)}z||,||z\otimes1-v_{e^2}z||,||z\otimes1-v_{e^4}z||\}.$$
\end{prp}
\begin{proof}
We write $z=x+y$ with $x\in H^0_1$ and $y\in H^0_2$. We put $z':=v_{(g)}z, x':=v_{(g)}x$. Notice that by the relation $(1)$ of Lemma \ref{bidule} we have $v_{(g)}x\in H_2\otimes K_{(g)}$. In particular $\langle x',x\otimes 1\rangle=0$ if $\langle\cdot,\cdot\rangle$ denotes the scalar product on $H\otimes K_{(g)}$. We have:
\begin{align}\label{uno}
\langle z\otimes1-x\otimes 1-x',x'\rangle=\langle z\otimes1-x',x'\rangle=\langle z\otimes1-z',x' \rangle
\end{align}
where the last equality comes from the fact that $\langle v_{(g)}y,x'\rangle=\langle y,x\rangle_2=0$ since $v_{(g)}$ is an isometry. Thus we have $$|\langle z\otimes1-x\otimes 1-x',x' \rangle|\le ||z\otimes1-z'||\ ||x'||=||z\otimes1-z'||\ ||x||_2.$$ We get
\begin{align*}
||x||^2_2+||y||^2_2&=||z||^2_2=||z\otimes1-x\otimes 1-x'+x\otimes1+x'||^2\\
&=||z\otimes1-x\otimes 1-x'||^2+||x||^2_2+||x'||^2+2Re\langle z\otimes1-x\otimes 1-x', x'\rangle\\
&\ge||z\otimes1-x\otimes 1-x'||^2+||x||^2_2+||x'||^2-2||z\otimes1-z'||\ ||x||_2
\end{align*}
by (\ref{uno}) and $\langle z\otimes1-x\otimes 1-x', x\otimes1\rangle=0$ (since $\langle x',x\otimes 1\rangle=0$ and $x\in H_1,y\in H_2$). Then we obtain 
\begin{equation}\label{tab}
||y||^2_2\ge||x||^2_2-2||x||_2||z\otimes1-v_{(g)}z||.
\end{equation}

Now, we consider the isometries $$v^2 : H\to H\otimes K_{e^2}\otimes K_{e^4} \text{ defined by } v^2\xi=(v_{e^2}\xi)_{12},$$
$$v^4 : H\to H\otimes K_{e^2}\otimes K_{e^4} \text{ defined by } v^4\xi=(v_{e^4}\xi)_{13},$$
 We have by $(2), (3)$ of Lemma \ref{bidule}
$$v^iH_2\subset H_1\otimes K_{e^2}\otimes K_{e^4},\ i=2,4,$$
$$v^2H_2\perp v^4H_2$$
We set $$Y_2=v^2y, Y_3=v^4y,$$ $$Z_2=v^2z, Z_3=v^4z.$$
Then, $Y_2, Y_3\in H_1\otimes K_{e^2}\otimes K_{e^4}$ and $Y_2, Y_3$ are orthogonal. Notice that this implies that the vectors $y\otimes 1\otimes 1$, $Y_2$ and $Y_3$ are pairwise orthogonal in $H\otimes K_{e^2}\otimes K_{e^4}$ since $y\in H_2$.

Consider $X:=z\otimes1\otimes1-y\otimes1\otimes1-Y_2-Y_3$ and notice that $X$ is orthogonal to $y\otimes1\otimes1$. Now we compute the scalar product:
$$\langle X,Y_2\rangle=\langle z\otimes1\otimes1-Y_2,Y_2\rangle=\langle z\otimes1\otimes1-Z_2,Y_2\rangle,$$
since $\langle v^2x,Y_2\rangle_2=\langle x,y\rangle_2=0$. Hence
$$|\langle X,Y_2\rangle|\le||y||_2||z\otimes1\otimes1-Z_2||.$$
Similarly, we have $$|\langle X,Y_3\rangle|\le||y||_2||z\otimes1\otimes1-Z_3||.$$
We obtain, since $\langle X,y\otimes1\otimes1\rangle=0$,
\begin{align*}
||x||_2^2+||y||_2^2&=||z||_2^2=||X+y\otimes1\otimes1+Y_2+Y_3||^2\\
&\ge||X||^2+3||y||_2^2-2||y||_2||z\otimes1\otimes1-Z_2||^2-2||y||_2||z\otimes1\otimes1-Z_3||^2
\end{align*}
and one can deduce that
\begin{equation}\label{deuzio}
||x||_2^2\ge2||y||_2^2-2||y||_2||z\otimes1-v_{e^2}z||-2||y||_2||z\otimes1-v_{e^4}z||.
\end{equation}

We denote by $A=||z\otimes1-v_{(g)}z||, B=||z\otimes1-v_{e^2}||$ and $C=||z\otimes1-v_{e^4}||$ and we get, applying (\ref{tab}), the fact that $||x||_2,||y||_2\le||z||_2$ and  (\ref{deuzio}),
\begin{align*}
||y||_2^2&\ge||x||_2^2-2||x||_2A\ge||x||_2^2-2||z||_2A\\
&\ge2||y||_2^2-2||z||_2(A+B+C).
\end{align*}
We then obtain $||y||_2^2\le2||z||_2(A+B+C).$

On the other hand, applying (\ref{deuzio}), the fact that $||x||_2,||y||_2\le||z||_2$ and  (\ref{tab}), we get
\begin{align*}
||x||_2^2\ge2||y||_2^2-2||z||_2(B+C)\ge2||x||_2^2-2||z||_2(2A+B+C).
\end{align*}

We then obtain $||x||_2\le2||z||_2(2A+B+C)$ and we can conclude:
$$||z||_2^2\le14||z||_2\max\{A,B,C\}.$$
\end{proof}

\begin{thm}
The $II_1$-factor $L^{\infty}(H_N^+(\Gamma))$ is full for all discrete groups $\Gamma$ and all $N\ge8$.
\end{thm}
\begin{proof}
This follows immediately from the previous proposition and the discussion before it.
\end{proof}


\begin{rque}
The fact that any bounded  asymptotically central sequence is asymptotically central implies that the center of $L^{\infty}(H_N^+(\Gamma))$ is trivial, and thus we get again that $L^{\infty}(H_N^+(\Gamma))$ is a factor for all $N\ge4$ and all discrete groups $\Gamma$.
\end{rque}
\subsection{Haagerup approximation property for the dual of $H_N^+(\Gamma)=\widehat{\Gamma}\wr_*S_N^+$, $\Gamma$ finite}\
In this subsection, we explain how to extend the main result of \cite{Lem13}, which deals with the case $\Gamma=\Z/s\Z$ with $s\in [1,+\infty)$, to a general finite group $\Gamma$. We assume $N\ge4$ and we set $\mathbb{G}=H_N^+(\Gamma)$. If $\alpha\in \text{Irr}(\mathbb{G})$, let $L_{\alpha}^2(\mathbb{G}):=\text{span}\{\Lambda_h(x) : \text{supp}(x)=\alpha\}\subset L^2(\mathbb{G})$ where $\Lambda_h$ is the GNS map associated to the GNS representation of the Haar state $h$ of $\mathbb{G}$.

Using the fundamental result in \cite[Theorem 3.7]{Bra12}, we can produce a net of normal, unital, completely positive $h$-preserving maps on $L^{\infty}(\mathbb{G})$ given by
$$T_{\phi_x}=\sum_{\alpha\in \text{Irr}(\mathbb{G})}\frac{\phi_x(\chi_{\overline{\alpha}})}{d_{\alpha}}P_{\alpha}.$$
In this formula, $\phi_x=ev_x\circ\pi$ with 
\begin{enumerate}
\item[$\bullet$] $ev_x$ the evaluation map in $x\in I_N=[4,N], N\ge5$ ($I_4=[0,4]$) on functions of $C(S_N^+)_0\simeq C([0,N])$, see \cite[Proposition 4.7, 4.8]{Bra12},
\item[$\bullet$] $\pi$ is the canonical map $\pi : C(H_N^+(\Gamma))\to C(S_N^+)$,
\item[$\bullet$] $P_{\alpha} : L^2(\mathbb{G})\to L^2(\mathbb{G})_{\alpha}$ is the orthogonal projection associated to $\alpha\in \text{Irr}(\mathbb{G})$.
\end{enumerate}

We now introduce a proper function on the monoid $S$ (see Theorem \ref{alter}). Let $L$ be defined by $L(\alpha)=\sum_{i=1}^{k_{\alpha}}l_i$ for $\alpha=a^{l_1}z_{g_1}\dots a^{l_{k_{\alpha}}}$ with $g_1,\dots,g_{k-1}\ne e$. Notice that, if $\Gamma$ is finite, for all $R>0$ the set $$B_R=\left\{\alpha=a^{l_1}z_{g_1}\dots a^{l_{k_{\alpha}}}: L(\alpha)=\sum_{i=1}^{k_{\alpha}}l_i\le R\right\}\subset S$$ is finite. Thus we get for a net $(f_{\alpha})_{\alpha}$: $$(f_{\alpha})_{\alpha\in S}\in c_0(S)\Longleftrightarrow\forall\epsilon>0\ \exists R>0: \forall\alpha\in S, \left(L(\alpha)>R\Rightarrow|f_{\alpha}|<\epsilon\right).$$ We say that a net $(f_{\alpha})_{\alpha}$ converges to $0$ as $\alpha\to\infty$ if $(f_{\alpha})_{\alpha}\in c_0(S)$.
One can prove, as in \cite[Proposition 3.3, Proposition 3.4]{Lem13}, that the net 
$\left(\dfrac{\phi_x(\chi_{\overline{\alpha}})}{d_{\alpha}}\right)_{x\in I_N}$ converges to $0$ as $\alpha\to\infty$ so that the extensions $T_{\phi_x} : L^2(H_N^+(\Gamma))\to L^2(H_N^+(\Gamma))$ are compact operators. The pointwise convergence to the identity of these operators, in $2$-norm, can be proved as in \cite[Theorem 3.5]{Lem13}. Then:

\begin{thm}
The dual of $H_N^+(\Gamma)=\widehat{\Gamma}\wr_* S_N^+$ has the Haagerup property for all finite groups $\Gamma$ and $N\ge4$.
\end{thm}

\section*{Acknowledgements}
I am very grateful to Roland Vergnioux for the time he spent discussing the arguments of this article. I would also like to thank Pierre Fima, Uwe Franz, Amaury Freslon and Moritz Weber for discussions, suggestions or commentaries which helped produce this article.

\bibliographystyle{alpha}
\nocite{*}
\bibliography{FWPHAP}

\end{document}